\documentclass[12pt]{amsart}

\usepackage{amsfonts, amsthm, amsmath}

\usepackage{amscd}

\usepackage[latin2]{inputenc}

\usepackage{t1enc}

\usepackage[mathscr]{eucal}

\usepackage{indentfirst}

\usepackage{graphicx}

\usepackage{graphics}

\usepackage{pict2e}

\usepackage{mathrsfs}

\usepackage{framed}

\usepackage{mathabx}
\usepackage{enumerate}
\usepackage[pagebackref]{hyperref}
\hypersetup{colorlinks=true}
\usepackage{cite}
\usepackage{color}
\usepackage{epic}
\usepackage{hyperref} 
\numberwithin{equation}{section}
\theoremstyle{plain}

\newtheorem{theorem}{Theorem}[section]

\newtheorem{lemma}[theorem]{Lemma}

\newtheorem{corollary}[theorem]{Corollary}

\newtheorem{proposition}[theorem]{Proposition}

\theoremstyle{definition}

\newtheorem{Def}[theorem]{Definition}

\newtheorem{example}[theorem]{Example}

\newtheorem{remark}[theorem]{Remark}

\newtheorem{?}[theorem]{Problem}

\def\inv{\mathrm{inv}}
\def\INV{\mathrm{INV}}
\def\MAJ{\mathrm{MAJ}}

\def\STAT{\mathrm{STAT}}
\def\BAST{\mathrm{BAST}}

\def\MIL{\mathrm{MIL}}
\def\F{\mathrm{F}}
\def\L{\mathrm{E}}

\def\des{\mathrm{des}}
\def\ides{\mathrm{ides}}

\def\bk{\mathrm{bk}}
\def\RGF{\mathrm{RGF}}
\def\rc{\mathrm{rc}}
\def\bMAJ{\mathrm{bMAJ}}
\def\bmajMIL{\mathrm{bmajMIL}}
\def\bmajBAST{\mathrm{bmajBAST}}
\def\bndes{\mathrm{bndes}}
\def\dinv{\mathrm{dinv}}
\def\area{\mathrm{area}}
\def\Val{\mathrm{Val}}

\def\ros{\mathrm{ros}}

\def\lb{\mathrm{lb}}
\def\ls{\mathrm{ls}}
\def\rb{\mathrm{rb}}
\def\rs{\mathrm{rs}}
\def\wrs{\mathrm{vrs}}
\def\wls{\mathrm{vls}}

\def\pl{\mathrm{pro}}
\def\Asc{\mathrm{Asc}}
\def\LL{\mathrm{L}}
\def\RR{\mathrm{R}}

\def\Des{\mathrm{Des}}
\def\Db{\mathrm{Db}}
\def\Id{\mathrm{Id}}
\def\st{\mathrm{st}}
\def\bst{\mathbf{st}}
\def\blb{\mathbf{lb}}
\def\bls{\mathbf{ls}}
\def\brb{\mathbf{rb}}
\def\brs{\mathbf{rs}}
\def\bwls{\mathbf{vls}}
\def\bwrs{\mathbf{vrs}}
\def\swap{\mathrm{swap}}
\def\suf{\mathrm{suf}}
\def\pre{\mathrm{pre}}

\def\st{\mathrm{st}}

\def\max{\mathrm{max}}
\def\min{\mathrm{min}}

\def\maj{\mathrm{maj}}

\def\SS{\mathcal{S}}

\def\RG{\mathcal{RG}}
\def\URG{\mathcal{URG}}
\def\OP{\mathcal{OP}}
\def\LD{\mathcal{LD}}

\newcommand\qbin[3]{\left[\begin{matrix} #1 \\ #2 \end{matrix} \right]_{#3}}

\def\boxit#1{\leavevmode\hbox{\vrule\vtop{\vbox{\kern.33333pt\hrule
    \kern1pt\hbox{\kern1pt\vbox{#1}\kern1pt}}\kern1pt\hrule}\vrule}}

\usepackage{collectbox}



\begin{document}

\title[From $q$-Stirling numbers to the Delta Conjecture]{From $q$-Stirling numbers to the Delta Conjecture: a viewpoint from vincular patterns
}

\author[J.~N.~Chen]{Joanna N. Chen}
\address{College of Science, Tianjin University of Technology, Tianjin 300384, P.R. China.}
\email{joannachen@tjut.edu.cn}

\author[S.~Fu]{Shishuo Fu}
\address{College of Mathematics and Statistics, Chongqing University, Chongqing 401331, P.R. China.}
\email{fsshuo@cqu.edu.cn}

\date{\today}

\begin{abstract}
The distribution of certain Mahonian statistic (called $\mathrm{BAST}$) introduced by Babson and Steingr\'{i}msson over the set of permutations that avoid vincular pattern $1\underline{32}$, is shown bijectively to match the distribution of major index over the same set. This new layer of equidistribution is then applied to give alternative interpretations of two related $q$-Stirling numbers of the second kind, studied by Carlitz and Gould. Moreover, extensions to an Euler-Mahonian statistic over ordered set partitions, and to statistics over ordered multiset partitions present themselves naturally. The latter of which is shown to be related to the recently proven Delta Conjecture. During the course, a refined relation between $\mathrm{BAST}$ and its reverse complement $\mathrm{STAT}$ is derived as well.
\end{abstract}


\maketitle

\tableofcontents

\section{Introduction}\label{sec:intro}
The $q$-Stirling numbers of the second kind were first studied by Carlitz
\cite{Ca1,Ca2} and Gould\cite{Go}. Following Gould\cite{Go}, we define two related $q$-Stirling numbers of the second kind, by recursions as follows,
\begin{equation}\label{def-sq1}
S_q(n,k) = \left\{
  \begin{array}{ll}
q^{k-1}S_q(n-1,k-1)+[k]_q S_q(n-1,k) & \mbox{if $0 < k \leq n$}\\[6pt]
1 & \mbox{if $n=k=0$}\\[6pt]
0  & \mbox{otherwise,}
 \end{array} \right.
\end{equation}
and
\begin{equation}\label{def-sq2}
\widetilde{S}_q(n,k) = \left\{
  \begin{array}{ll}
\widetilde{S}_q(n-1,k-1)+[k]_q \widetilde{S}_q(n-1,k) & \mbox{if $0 < k \leq n$}\\[6pt]
1 & \mbox{if $n=k=0$}\\[6pt]
0  & \mbox{otherwise,}
 \end{array} \right.
\end{equation}
where
\begin{equation*}
  [k]_q=1+q+\cdots+q^{k-1}.
\end{equation*}
Note that
\begin{equation*}
  S_q(n,k)=q^{\binom{k}{2}} \widetilde{S}_q(n,k).
\end{equation*}

In 1978, Milne \cite{Mi} defined $S_q(n,k)$ combinatorially by introducing
certain statistic on set partitions, which we shall call $\MIL$.
Since then, many authors such as Garsia and Remmel \cite{GR}, Sagan \cite{Sa}, Wachs and White \cite{WW}, White \cite{Wh}, Ehrenborg and Readdy \cite{ER}, Wagner \cite{Wa} have studied statistics with distributions given by $S_q(n,k)$ or $\widetilde{S}_q(n,k)$. Most of them were defined on set partitions or \emph{restricted growth functions}, which are frequently used to encode set partitions.

The concept of restricted growth function was introduced in \cite{Hu,Wi} and the name was coined by Milne \cite{Mil}. Let $[n]:=\{1,2, \cdots, n\}$. Given a word $w$ in $[k]^n$, it is said to be a restricted growth function or an $\RGF$ if
$$w_1=1, \text{ and } w_i \leq \max\{w_j: 1 \leq j <i\} +1, \text{ for } 2\le i\le n.$$

Wachs and White \cite{WW} defined four inversion-like statistics on RGFs, namely, $\ls$, $\rs$, $\lb$ and $\rb$. Let $\RG(n,k)$ be the set of $\RGF$s of length $n$ with $\bk(w):=\max\{w_i:1\le i\le n\}=k$, and $\RG(n)$ be the set of all $\RGF$s of length $n$. By Milne\cite{Mi} and Wachs and White \cite{WW}, we have the following theorem.

\begin{theorem}\label{thm-WW}
For $n \geq k \geq 1$, we have
\begin{equation}\label{equ-ls}
  \sum_{w \in \RG(n,k)}q^{\ls(w)}=\sum_{w \in \RG(n,k)}q^{\rb(w)}=S_q(n,k),
\end{equation}
and
\begin{equation}\label{equ-rs}
  \sum_{w \in \RG(n,k)}q^{\lb(w)}=\sum_{w \in \RG(n,k)}q^{\rs(w)}=\widetilde{S}_q(n,k).
\end{equation}
\end{theorem}

In this paper, we introduce two new statistics $\wls$ and $\wrs$ on $\RGF$s, whose distributions are
given by $S_q(n,k)$ and $\widetilde{S}_q(n,k)$, respectively.

\begin{theorem} \label{main-rgf}
For $n \geq k \geq 1$, we have
\begin{equation}\label{equ-wls}
  \sum_{w \in \RG(n,k)}q^{\wls(w)}=S_q(n,k),
\end{equation}
and
\begin{equation}\label{equ-wrs}
  \sum_{w \in \RG(n,k)}q^{\wrs(w)}=\widetilde{S}_q(n,k).
\end{equation}
\end{theorem}

Combining \eqref{equ-ls} and \eqref{equ-wls}, we see that statistics
$\ls$ and $\wls$ have the same distribution on $\RG(n,k)$.
Combining \eqref{equ-rs} and \eqref{equ-wrs}, we see that statistics
$\rs$ and $\wrs$ have the same distribution on $\RG(n,k)$.
 These facts will be proved via a bijection $\xi$ and a bijection $\zeta$, which also preserve some other statistics.

\begin{theorem}\label{main-bijection}
Statistics $(\LL, \Asc, \ls)$ and
$(\LL, \Asc, \wls)$ have the same joint distribution on $\RG(n)$ for all $n \geq 1$.
\end{theorem}

\begin{theorem}\label{main-bijection-r}
Statistics $(\LL, \Asc, \rs)$ and
$(\LL, \Asc, \wrs)$ have the same joint distribution on $\RG(n)$ for all $n \geq  1$.
\end{theorem}

The motivation for these two new statistics $\wls$ and $\wrs$ stems from their counterparts defined on pattern-avoiding permutations. We use $\SS_n(1\underline{32})$ to denote the set of permutations of $[n]$ avoiding the vincular pattern $1\underline{32}$, and $\SS_n^k(1\underline{32})$ to denote its subset that is composed of permutations with $k$ descents. And we use $\SS_n$ (resp. $\SS_n^k$) to denote the bigger set with the pattern avoidance condition lifted. As pointed out by Claesson \cite[Proposition~3]{Cl}, there is a natural one-to-one correspondence between permutations in $S_n(1\underline{32})$ with $k-1$ descents and set partitions of $[n]$ with $k$ blocks, which in turn, can be bijectively mapped to $\RG(n,k)$ (see for example \cite[Theorem~5.1]{SW}). Consequently, we have the following three corollaries.
\begin{corollary}\label{main-Sn}
For $n \geq k \geq 1$, we have
\begin{equation*}
  \sum\limits_{\sigma \in \SS_n^{k-1}(1\underline{32})}
    q^{\MAJ(\sigma)}
=
\sum\limits_{\sigma \in \SS_n^{k-1}(1\underline{32})}
    q^{\BAST(\sigma)}=S_q(n,k),
\end{equation*}
and
\begin{equation*}
   \sum\limits_{\sigma \in \SS_n^{k-1}(1\underline{32})}
q^{2\underline{13}(\sigma)}
=
\sum\limits_{\sigma \in \SS_n^{k-1}(1\underline{32})}
q^{2\underline{31}(\sigma)}=\widetilde{S}_q(n,k).
\end{equation*}
\end{corollary}

\begin{corollary}\label{main-bijectionper}
Statistics $(\Db, \Id, \MAJ)$ and
$(\Db, \Id, \BAST)$ have the same joint distribution on $\SS_n(1\underline{32})$ for all $n \geq 1$. (See Table~\ref{MAJ-BAST} for the case of $n=4$.)
\end{corollary}

\begin{corollary}\label{main-bijectionper-r}
Statistics $(\Db, \Id, 2\underline{13})$ and
$(\Db, \Id, 2\underline{31})$ have the same joint distribution on $\SS_n(1\underline{32})$ for all $n\geq 1$.
\end{corollary}


We will extend the set of $\RGF$s to the set of \emph{unrestricted growth functions}, or \emph{URGs} for short, in order to encode the set of \emph{ordered partitions}. These are partitions whose blocks can be arbitrarily permuted, in contrast to the ``least element increasing'' convention for the (unordered) partitions. We take $\URG(n)$ to be the set of URGs of length $n$, and $\URG(n,k)$ to be the set of all $w\in\URG(n)$ with $\bk(w)=k$. Let $[k]_q!:=[1]_q[2]_q\cdots [k]_q$.
Following Steingr\'imsson \cite{Ste}, the statistics defined on $\URG(n,k)$ with distribution given by $[k]_q!\cdot S_q(n,k)$ are called \emph{Euler-Mahonian}.
Steingr\'imsson introduced a statistic $\bmajMIL$ and showed that it is Euler-Mahonian. Encouraged by Theorem~\ref{main-bijection} and Corollary~\ref{main-bijectionper}, we introduce a new Euler-Mahonian statistic $\bmajBAST$, defined on URGs.

\begin{theorem}\label{main-URG}
For $n\ge k\ge 1$, we have
\begin{align*}
\sum_{w\in\URG(n,k)}q^{\bmajMIL(w)}=\sum_{w\in\URG(n,k)}q^{\bmajBAST(w)}=[k]_q!\cdot S_q(n,k).
\end{align*}
\end{theorem}

The final objects we would like to consider are the ordered multiset partitions, whose definition we postpone to Section~\ref{subsec:multiset}. The interests stem from the recently proven Delta Conjecture (see for example \cite{HRW,Rh,HRS,BCH}), the Valley Version of which asserts the following combinatorial formula for the quasisymmetric function
\begin{align}\label{DeltaConj}
\Delta'_{e_k}e_n=\Val_{n,k}(x;q,t):=\{z^{n-k-1}\}\left[\sum_{P\in\LD_n}q^{\dinv(P)}t^{\area(P)}\prod_{i\in\Val(P)}(1+z/q^{d_i(P)+1})x^P\right].
\end{align}
We will not need the operator $\Delta'_f$, the set $\LD_n$ of labeled Dyck paths, nor the other undefined notations appearing in \eqref{DeltaConj}; for details on them see \cite{HRW}. We will extend both $\bmajMIL$ and $\bmajBAST$ to ordered multiset partitions and establish

\begin{theorem}\label{main-OMP}
For all $n>k\ge0$, we have
\begin{align}\label{id1:big Val-q-0}
q^{\binom{k+1}{2}}\Val_{n,k}(x;q,0)&=q^{\binom{k+1}{2}}\sum_{\beta\models_0 n}\sum_{\mu\in\OP_{\beta,k+1}}q^{\inv(\mu)}x^{\beta}\\
&=\sum_{\beta\models_0 n}\sum_{\mu\in\OP_{\beta,k+1}}q^{\bmajMIL(\mu)}x^{\beta}=\sum_{\beta\models_0 n}\sum_{\mu\in\OP_{\beta,k+1}}q^{\bmajBAST(\mu)}x^{\beta}.\label{id2:big Val-q-0}
\end{align}
\end{theorem}
\begin{remark}\label{rmk:inv}
The $\inv(\mu)$ appearing in \eqref{id1:big Val-q-0} is a generalization of the inversion number on words to ordered multiset partitions (see \cite{Wils,HRW,RW,Rh}). When restricted to ordered set partitions, $\inv$ is exactly the statistic $\ros$ defined by Steingr\'imsson \cite{Ste}, and is equivalent to $\lb$ in \eqref{equ-rs} when we use URGs instead. The proof of \eqref{id1:big Val-q-0} first appeared in \cite[Prop.~4.1, Eq.~(49)]{HRW}.
\end{remark}

The paper is organized as follows. In the next section, we will first present the formal definitions of all the statistics that concern us here, and the meanings of vincular pattern and $\mathrm{URG}$, then we go on to explain the transition from Theorems~\ref{thm-WW}, \ref{main-rgf}, \ref{main-bijection} and \ref{main-bijection-r} to Corollaries~\ref{main-Sn}, \ref{main-bijectionper} and \ref{main-bijectionper-r}. An algebraic proof of Theorem~\ref{main-rgf} will be given in Section~\ref{sec:algebra}. The aforementiond bijection $\xi$ will be constructed in Section~\ref{sec:involution}, to give a proof of Theorem~\ref{main-bijection}.
Another bijection $\zeta$ will be given in Section~\ref{sec:bijection}, which leads to a proof of Theorem~\ref{main-bijection-r}. An involution $\psi$ on $\SS_n$ will be defined in Section~\ref{sec:stat-bast}, to reveal a finer relation between $\BAST$ and $\STAT$. Theorem~\ref{main-URG} will be proved and then strengthened (via the involution $\psi$) in Section~\ref{subsec:OP}, and the stronger Theorem~\ref{main-URG-st} will then enable us to prove Theorem~\ref{main-OMP} in Section~\ref{subsec:multiset}. We conclude with several remarks in the final section.


\section{Preliminaries}\label{sec:pre}
In this section, we are going to recall or introduce four different but related combinatorial objects, namely, RGFs, pattern-avoiding permutations, URGs and barred permutations, along with various statistics defined on them.

Given any word $w=w_1 w_2 \cdots w_n$ (not necessarily an $\RGF$),
we define $\Asc(w) =\{i: w_i < w_{i+1}\}$ to be
the set of the positions of \emph{ascents}. We make the convention here, that if $\st_i(w)$ is certain coordinate statistic defined for each $i\in[n]$, then the bold-faced $\bst(w):=(\st_1(w)~\cdots~\st_n(w))$ and $\st(w):=\sum_i\st_i(w)$ stand for the corresponding vector statistic and global statistic, respectively.
Wachs and White\cite{WW} gave the following definitions. Let
\begin{align*}
  \LL(w)=&\{i \in [n]:~ w_i~ \text{is the leftmost occurrence of}~ w_i\},\\[5pt]
   \RR(w)=&\{i \in [n]:~ w_i~ \text{is the rightmost occurrence of}~ w_i\}.
\end{align*}

For $i \in [n]$, let
\begin{align*}
   \lb_i(w)=&\#~\{j \in \LL(w): j<i ~\text{and}~ w_j > w_i\}, \\[5pt]
  \ls_i(w)=&\#~\{j \in \LL(w): j<i ~\text{and}~ w_j < w_i\}, \\[5pt]
  \rb_i(w)=&\#~\{j \in \RR(w): j>i ~\text{and}~ w_j > w_i\}, \\[5pt]
  \rs_i(w)=&\#~\{j \in \RR(w): j>i ~\text{and}~ w_j < w_i\}.
\end{align*}


As an example, for $w=12134243221435322$, we see that
\begin{align*}
  \LL(w) &=\{1,2,4,5,14\},\\[5pt]
  \RR(w) &=\{11,12,14,15,17\},\\[5pt]
  \blb(w)&= (0~0~1~0~0~2~0~1~2~2~3~0~1~0~2~3~3), \\[5pt]
  \bls(w)&= (0~1~0~2~3~1~3~2~1~1~0~3~2~4~2~1~1), \\[5pt]
  \brb(w)&= (4~3~4~2~1~3~1~2~3~3~4~1~1~0~0~0~0), \\[5pt]
  \brs(w)&= (0~1~0~2~3~1~3~2~1~1~0~2~1~2~1~0~0).
\end{align*}
Consequently, we have $\lb(w)=20$, $\ls(w)=27$, $\rb(w)=32$ and $\rs(w)=20$.

Now we are ready to give descriptions of $\wls$ and $\wrs$, which can be seen as variants of $\ls$ and $\rs$.
Let $\pl(w)$ be the {\bf p}osition of the {\bf r}ightmost {\bf o}ne in $w$.
For $i \in [n]$, let
\begin{align*}
  \wls_i(w)&=\left\{
  \begin{array}{ll}
\ls_i(w)-1  \,\,\,\,\,\,\,\,\,  \mbox{if $i \notin \RR(w)$ and $i > \pl(w)$}\\[5pt]
\ls_i(w)+1  \,\,\,\,\,\,\,\,\,\mbox{if $i \in \RR(w)$ and $i < \pl(w)$}\\[5pt]
\ls_i(w)  \,\,\,\,\,\,\,\,\,\,\,\,\,\,\,\,\,\,\,  \mbox{otherwise,}
 \end{array} \right. \\[5pt]
   \wrs_i(w)&=\left\{
  \begin{array}{ll}
\rs_i(w)+1  \,\,\,\,\,\,\,\,\,  \mbox{if $i \notin \RR(w)$ and $i > \pl(w)$}\\[5pt]
\rs_i(w)-1  \,\,\,\,\,\,\,\,\,\mbox{if $i \in \RR(w)$ and $i < \pl(w)$}\\[5pt]
\rs_i(w)  \,\,\,\,\,\,\,\,\,\,\,\,\,\,\,\,\,\,\,  \mbox{otherwise.}
 \end{array} \right.
\end{align*}


For our running example $w=12134243221435322$, we have $n=17$, $k=5$, $\pl(w)=11$, and
\begin{align*}
  \bwls(w)&= (0~1~0~2~3~1~3~2~1~1~0~3~1~4~2~0~1), \\[5pt]
  \bwrs(w)&= (0~1~0~2~3~1~3~2~1~1~0~2~2~2~1~1~0).
\end{align*}
It follows that $\wls(w)=25$ and $\wrs(w)=22$.
In fact, it is evident from the definition that for any word $w$ and any $i\in[n]$, we have
\begin{align}\label{wls+wrs}
\wls_i(w)+\wrs_i(w)=\ls_i(w)+\rs_i(w).
\end{align}
A less obvious relation is
\begin{align}\label{def-wls}
\ls(w)-\wls(w)=n-\pl(w)-\bk(w)+1.
\end{align}
One can check this with the former example, $\ls(w)-\wls(w)=2=17-11-5+1$.

An occurrence of a classical \emph{pattern} $p$ in a permutation $\sigma$ is a subsequence of $\sigma$ that is order-isomorphic to $p$. For example, $41253$ has two occurrences of the pattern $3142$, as witnessed by its subsequences $4153$ and $4253$. $\sigma$ is said to \emph{avoid} $p$ if there exists no occurrence of $p$ in $\sigma$. The readers are highly recommended to see Kitaev's book \cite{Kit} and the references therein for a comprehensive introduction to patterns in permutations. We follow \cite{Kit} for most of the upcoming notations.

Recall that a permutation statistic is called \emph{Mahonian} if it has the same distribution with the number of inversions, denoted $\INV$, over $\SS_n$. In an effort to characterize various Mahonian statistics, Babson and Steingr\'{i}msson \cite{BS} generalized the notion of permutation patterns, to what are now known as \emph{vincular patterns} \cite[Chap.~1.3 and Chap.~7]{Kit}. Adjacent letters in a vincular pattern which are
 underlined must stay adjacent when they are placed  back to the original permutation. 
 For comparison, $41253$ now contains only one occurrence of the vincular pattern $\underline{31}42$ in its subsequence $4153$, but not in $4253$ any more. Given a vincular pattern $\tau$ and a permutation $\pi$, we denote by $\tau(\pi)$ the number of occurrences of the pattern $\tau$ in $\pi$, and $(\tau_1+\tau_2)(\pi):=\tau_1(\pi)+\tau_2(\pi)$.

  Babson and Steingr\'{i}msson \cite{BS} showed that most of the Mahonian statistics in the literature can be expressed as the sum of vincular pattern functions. We list some of them below.
\begin{align*}
\INV &=\underline{21}+3\underline{12}+3\underline{21}+2\underline{31},\\
\MAJ &=\underline{21}+1\underline{32}+2\underline{31}+3\underline{21},\\
\STAT &=\underline{21}+\underline{13}2+\underline{21}3+\underline{32}1,\\
\BAST &=\STAT^{\rc}=\underline{21}+2\underline{13}+1\underline{32}+3\underline{21},
\end{align*}
where $\rc$ stands for the function composition of the reversal $r$ and the complement $c$. Given a permutation $p=p_1 p_2 \cdots p_n$, recall that
\begin{align*}
r(p_1p_2\cdots p_n) &=p_np_{n-1}\cdots p_1,\\
c(p_1p_2\cdots p_n) &=(n+1-p_1)(n+1-p_2)\cdots(n+1-p_n).
\end{align*}
Moreover, let
\begin{align*}
\Des(p)&=\{i:p_i > p_{i+1}\}, \quad \des(p)= \sum_{j\in \Des(p)}1,\\[3pt]
\Db(p)&=\{p_1\}\cup\{p_{i+1}:p_i>p_{i+1}, 1\leq i<n\},\\[3pt]
\Id(p)&= \Des(p^{-1}), \quad
\ides(p) = \sum_{j\in \Id(p)}1.
\end{align*}
If we use the standard two-line notation to write $p$, then $p^{-1}$ is obtained by switching the two lines and rearranging the columns to make the first line increasing. For instance, if $p=\left(\begin{array}{ccc} 1 & 2 & 3\\ 3 & 1 & 2\end{array}\right)$, then $p^{-1}=\left(\begin{array}{ccc} 1 & 2 & 3\\ 2 & 3 & 1 \end{array}\right)$.

Actually there exists a stronger relation between the two Mahonian statistics $\MAJ$ and $\STAT$. In the same paper \cite{BS}, Babson and Steingr\'{i}msson conjectured the bistatistic $(\des,\STAT)$ is equidistributed with $(\des,\MAJ)$. This was first proved by Foata and Zeilberger \cite{FZ}, see also \cite{Bur,KV,CL,FHV} for further developements along this line. Since clearly $\des(p^{\rc})=\des(p)$ for any permutation $p$, we include here the equivalent version for $\BAST$, which will also be needed in Section~\ref{sec:stat on OP}.
\begin{proposition}[Theorem~3 in \cite{FZ}]
For $n>k\ge 0$, we have
\begin{align}
\label{des-maj=des-bast}
\sum_{p\in\SS_n^k}q^{\MAJ(p)}=\sum_{p\in\SS_n^k}q^{\BAST(p)}.
\end{align}
\end{proposition}

The following relation parallels \eqref{def-wls}, and has previously been observed in \cite[Lemma~5.4]{FTHZ}. For any permutation $p\in\SS_n$,
\begin{align}\label{231-213}
2\underline{31}(p)-2\underline{13}(p)=n-p_n-\des(p)=\MAJ(p)-\BAST(p).
\end{align}

Wachs~\cite{Wac} introduced the more general $\sigma$-$\RGF$s, to encode the set of ordered partitions. But the following notion of unrestricted growth function seems to be more appropriate for our purpose.
\begin{Def}
Given a word $w\in [k]^n$, it is said to be an unrestricted growth function or a $\mathrm{URG}$ if for any $j$, $1\le j \le \bk(w)$, $j$ appears in $w$.
\end{Def}
For a given ordered set partition of $[n]$, say $\pi=B_1/B_2/\cdots/B_k$, we form a word $w=w_1\cdots w_n$ by taking $w_i=j$, if and only if $i\in B_j$. This is clearly seen to be a bijection between ordered partitions of $[n]$ into $k$ blocks, and $\URG(n,k)$.
\begin{example}\label{URG-OP}
There are in total six ordered partitions of $\{1,2,3\}$ into two blocks, and also six words in $\URG(3,2)$. We list them below in one-to-one correspondence, with the first three being the (unordered) ones in $\RG(3,2)$.
$$
\begin{array}{c|c|c|c|c|c}
1~2~/~3~ & ~1~3~/~2~ & ~1~/~2~3~ & ~3~/~1~2~ & ~2~/~1~3~ & ~2~3~/~1~ \\[5pt]
$112$ & $121$ & $122$ & $221$ & $212$ & $211$
\end{array}
$$
\end{example}

Besides the total order between all the letters of a given word $w$, imposed by their numerical values, we also need to consider a \emph{partial order} $\succ$. By $a\succ b$ we mean there exists some $w_i=a$ and $w_j=b$, with $i\in\LL(w)$, $j\in\RR(w)$ and $i>j$.

Now we can recall the statistics $\MIL, \bMAJ$ and $\bmajMIL$ used in \cite{Ste}, reformulated in terms of words, as well as a new statistic \emph{block nondescents}, denoted by $\bndes$. For any word $w$, we let
\begin{align*}
\MIL(w) &=\sum_i (w_i-1),\\
\bMAJ(w) &=\sum_i i\chi(i\succ i+1),\\
\bmajMIL(w) &=\bMAJ(w)+\MIL(w),\\
\bndes(w) &=\bk(w)-1-\sum_i \chi(i\succ i+1),
\end{align*}
where $\chi(S)=1$ (resp.~$\chi(S)=0$) if the statement $S$ is true (resp. false).

Note that although for each $w\in\RG(n)$,
\begin{align}\label{mil=ls,bndes=k-1}
\MIL(w)=\ls(w), \text{ and }\; \bndes(w)=\bk(w)-1,
\end{align}
in general $\MIL(w)\neq \ls(w)$, $\bndes(w)\neq\bk(w)-1$ for $w\in\URG(n)$. Take $w=22311$ for example, we see $\MIL(w)=4\neq\ls(w)=1$, $\bndes(w)=1\neq\bk(w)-1=2$.

As noted by Claesson \cite[Proposition~3]{Cl} or even earlier, there is a natural bijection sending a set partition of $[n]$ with $k$ blocks, to a $1\underline{32}$-avoiding permutation of $[n]$ with $k-1$ descents. Combining this with
the bijection between set partitions and $\RGF$s, we may construct
a bijection $\theta$ directly from $\SS_n(1\underline{32})$ to $\RG(n)$ that sends $\des+1$ to $\bk$. However, in order to deal with URGs and prove Theorem~\ref{main-URG}, we define our bijection $\theta$ on a bigger set, called \emph{the barred permutations}, whose definition we give below.
\begin{Def}
A barred permutation $\widebar{p}$ is a permutation $p$ decorated with certain amount of bars in between its consecutive letters, according to the following rules.
\begin{itemize}
	\item If $p_i>p_{i+1}$, then we must insert a bar between $p_i$ and $p_{i+1}$. We call such bar a fixed bar and use the double line $\parallel$ to represent it, although it still counts as one bar for enumeration.
	\item If $p_i<p_{i+1}$, then we may or may not insert a bar between $p_i$ and $p_{i+1}$. We call such bar an active bar and use the single line $\mid$ to represent it.
\end{itemize}
Permutation $p$ is then called the base permutation of $\widebar{p}$. For $n>a+b\ge 0$, $a\ge 0$ and $b\ge 0$, we denote by $\widebar{\SS}_n$ the set of all barred permutations of $n$, and by $\widebar{\SS}_n^{a,b}$ the subset with exactly $a$ active bars and $b$ fixed bars.
\end{Def}

We note that $\SS_n^k$ can be embedded naturally in $\widebar{\SS}_n$ as the subset $\widebar{\SS}_n^{0,k}$, simply by placing a bar at every descent. Consequently, the preimage of each $\widebar{p}\in\widebar{\SS}_n^{0,k}$ is exactly its base permutation $p$. For this reason, we will use both $p$ and $\widebar{p}$ interchangeably when it has no active bars.

Given a barred permutation $\widebar{p}\in\widebar{\SS}_n$, we now construct a $\mathrm{URG}$ $w=\theta(\widebar{p})=w_1\cdots w_n$, such that
$$w_i=j \text{ if and only if there are $j-1$ bars to the right of $i$ in $\widebar{p}$} .$$
\begin{example}
The six $\mathrm{URG}$s listed in Example~\ref{URG-OP} are reproduced below, as the images of six barred permutations, under the map $\theta$.
$$
\begin{array}{c|c|c|c|c|c}
\phantom{wwwww}\widebar{p}: \quad 3\parallel 12 & 2\parallel 13 & 23\parallel 1 & 12\mid 3 & 13\parallel 2 & 1\mid 23 \\[5pt]
w=\theta(\widebar{p}):\quad $112$ & $121$ & $122$ & $221$ & $212$ & $211$
\end{array}
$$
\end{example}

\begin{lemma}\label{lem-rgf2per}
For all $n\ge 1$, the map $\theta$ as described above is a bijection between $\widebar{\SS}_n$ and $\URG(n)$. It induces a bijection between $\SS_n(1\underline{32})$ and $\RG(n)$. Moreover, for each $p\in\SS_n(1\underline{32})$, suppose $w=\theta(p)$, then we have
 \begin{align}\label{per-RG}
 (\Db,\Id,\MAJ,2\underline{13},\BAST,2\underline{31})~p = (\LL,\Asc,\ls,\rs,\wls,\wrs)~ w, \text{ and }~p_n=\pl(w).
 \end{align}
\end{lemma}
\begin{proof}
By our construction, the image of each barred permutation in $\widebar{\SS_n}$ is clearly in $\URG(n)$, hence $\theta$ is well-defined. The injectivity of $\theta$ is also clear. The inverse of $\theta$ can be described as reading $w$ from left to right, and recording the positions of its largest letters $\bk(w)$, then putting a bar, next recording the positions of $\bk(w)-1$, putting a bar, and so on and so forth. Therefore $\theta$ is indeed a bijection.

When we restrict $\theta$ so that the image set is $\RG(n)$, then the defining condition for restricted growth functions, namely, $$w_{i+1}\le \max\{w_1,\ldots,w_i\}+1$$ forces $p$ to avoid pattern $1\underline{32}$, and $\widebar{p}$ to be without any active bars, and vice versa.

It remains to show the equivalence of those statistics. The first four $$(\Db,\Id,\MAJ,2\underline{13})~p=(\LL,\Asc,\ls,\rs)~w, \text{ and }~p_n=\pl(w)$$ can be easily checked using the definitions. In particular, $\Id(p)=\Asc(w)$ and $p_n=\pl(w)$ hold for all $\widebar{p}\in\widebar{\SS}_n$. Thanks to \eqref{wls+wrs}, \eqref{def-wls} and \eqref{231-213}, the equivalence of the remaining two is somewhat routine and thus omitted.
\end{proof}
In view of \eqref{per-RG},
Corollary \ref{main-Sn} follows  immediately from
Theorem \ref{thm-WW} and  Theorem \ref{main-rgf},
while Corollaries \ref{main-bijectionper} and \ref{main-bijectionper-r} follow from
Theorems \ref{main-bijection} and \ref{main-bijection-r}, respectively.

The observations \eqref{def-wls}, \eqref{231-213}, \eqref{mil=ls,bndes=k-1} and Lemma~\ref{lem-rgf2per}, motivated us to generalize the permutation statistic $\BAST$ to a statistic for $\mathrm{URG}$s.

For any $w\in\URG(n)$, we let
\begin{align}\label{def-bmajbast}
\bmajBAST(w):=\bmajMIL(w)+\pl(w)+\bndes(w)-n.
\end{align}

\section{An algebraic proof of Theorem \ref{main-rgf} }\label{sec:algebra}
 In this section, we will utilize the following recurrence given by Wagner \cite{Wa1996}, to give an algebraic proof of Theorem \ref{main-rgf}.

 \begin{lemma}
 For $n \geq k \geq 1$, we have
 \begin{align}
   \widetilde{S}_q(n+1,k)=\sum_{j=0}^n {\binom{n}{j}} q^{j-k+1} \widetilde{S}_q(j,k-1).\label{recur2}
 \end{align}
 \end{lemma}

\begin{proof}[Proof of Theorem~\ref{main-rgf}]
We first give a proof of (\ref{equ-wls}), which is by induction on $n$. We assume that (\ref{equ-wls}) is true for $n-1$.
Given $w=w_1 w_2 \cdots w_n \in \RG(n,k) $, we consider the following two cases.
\begin{itemize}
  \item $w \in S:=\{w\in \RG(n,k):n \in L(w)\}$.

  Assume that  $w'=w_1 w_2 \cdots w_{n-1}$. Clearly, $w'$ is an $\RGF$ of length $n-1$ with $\bk(w')=k-1$.
  By the alternative definition of $\wls$ given in
  (\ref{def-wls}) and the fact that $\ls(w)=\sum_{i=1}^n(w_i-1)$, it is not hard to see that
  $\wls(w)=\wls(w')+k-1$.
  Thus, we have
  \begin{equation*}
    \sum_{w \in S}q^{\wls(w)}=q^{k-1}\sum_{w' \in \RG(n-1,k-1) } q^{\wls(w')}=q^{k-1} S_q(n-1,k-1).
  \end{equation*}

  \item $w \in T:=\{w\in \RG(n,k):n \notin L(w)\}$.

 Let $w'=w_1 w_2 \cdots w_{n-1}$. Since $n \notin L(w)$,
 we see that $\bk(w')=\bk(w)=k$. If $1< w_n \leq k$, it is easily seen that $\ls(w)=\ls(w')+w_n-1$ and
 $\pl(w)=\pl(w')$. Thus, we apply \eqref{def-wls} again to deduce that
 $\wls(w)=\wls(w')+w_n-2$. If $w_n=1$, then $\pl(w)=n$, by definition we have
 \begin{align*}
   \wls(w) &=\ls(w)+\bk(w)-1=\ls(w')+k-1.
 \end{align*}
 It follows that
 \begin{align*}
   \sum_{w \in T}q^{\wls(w)} & = \sum_{w \in T,\, w_n=1}q^{\wls(w)}+\sum_{w \in T,\, w_n \neq 1}q^{\wls(w)}\\[6pt]
   & =q^{k-1} \sum_{w' \in \RG(n-1,k)}q^{\ls(w')}+(1+q+\cdots+q^{k-2})\sum_{w' \in \RG(n-1,k)}q^{\wls(w')}.
 \end{align*}
 By (\ref{equ-ls}) and the induction hypothesis, we have
 \begin{align*}
   \sum_{w \in T}q^{\wls(w)}&=(1+q+\cdots+q^{k-2}+q^{k-1})S_q(n-1,k)\\[3pt]
   &=[k]_qS_q(n-1,k).
 \end{align*}
\end{itemize}
Combining the above two cases and (\ref{def-sq1}) , we have
\begin{align*}
 \sum_{w \in \RG(n,k)}q^{\wls(w)} &=\sum_{w \in S}q^{\wls(w)} +\sum_{w \in T}q^{\wls(w)} \\[5pt]
   & =q^{k-1}S_q(n-1,k-1)+[k]_q S_q(n-1,k)\\[5pt]
   &=S_q(n,k).
  \end{align*}
This completes the proof of (\ref{equ-wls}).

Now we proceed to give a proof of (\ref{equ-wrs}).
Given an $\RGF$ $v=v_1 v_2 \cdots v_{n+1} \in \RG(n+1,k)$,
we see that $v_1=1$.
Let $v'$ be the $\RGF$ obtained from $v$ by
deleting all the $1$'s in $v$ and decreasing
each remaining letter by $1$. Suppose that the number of occurrences of $1$ in $v$ is $n-j+1~(0 \leq j \leq n)$. Clearly, $v'$
is an $\RGF$ with length $j$ and maximum letter $k-1$. Conversely, given such a $v'\in\RG(j,k-1)$, there are $\binom{n}{j}$ different ways to insert $n-j+1$ $1$'s to recover certain $v\in\RG(n+1,k)$.

For an index $i, 1\le i\le n+1$. If $i>\pl(v)$ and $i\in\RR(v)$, then $\wrs_i(v)=\rs_i(v)$, which equals the contribution to $\rs(v')$ from $v_i$. If $i>\pl(v)$ and $i\not\in\RR(v)$, then $\wrs_i(v)-1=\rs_i(v)$, which equals the contribution to $\rs(v')$ from $v_i$. The remaining cases can be discussed similarly, which amount to giving
\[\wrs(v)=\rs(v')+j-(k-1).\]
It follows that
\begin{align*}
  \sum_{v \in \RG(n+1,k)} q^{\wrs(v)}& = \sum_{j=0}^n  \binom{n}{j}  \sum_{v' \in \RG(j,k-1)} q^{\rs(v')+j-(k-1)}\\[5pt]
   &=\sum_{j=0}^n \binom{n}{j}  q^{j-k+1}\sum_{v' \in \RG(j,k-1)} q^{\rs(v')}.
\end{align*}
By (\ref{equ-rs}) in Theorem \ref{thm-WW},
we have
\begin{equation*}
  \sum_{v' \in \RG(j,k-1)} q^{\rs(v')}=\widetilde{S}_q(j,k-1).
\end{equation*}
By (\ref{recur2}), it follows that
\begin{align*}
  \sum_{v \in \RG(n+1,k)} q^{\wrs(v)} & =\sum_{j=0}^n \binom{n}{j}  q^{j-k+1} \widetilde{S}_q(j,k-1) \\[5pt]
   & = \widetilde{S}_q(n+1,k).
\end{align*}
This completes the proof of (\ref{equ-wrs}).
\end{proof}

\section{A bijective proof of Theorem~\ref{main-bijection}}\label{sec:involution}
We need two local operators $\delta^+$ and $\delta^-$, which are crucial in the construction of $\xi$ and its inverse $\xi^{-1}$. For any word $w$ composed of nonnegative integers, take any letter $x\in w$, we define
\begin{align*}
\delta^+(x) &=\begin{cases}
x & \text{if $x$ is a left-to-right maximum in $w$},\\
x+1 & \text{otherwise},
\end{cases}
\\
\delta^-(x) &=\begin{cases}
x & \text{if $x$ is a left-to-right maximum in $w$},\\
x-1 & \text{otherwise}.
\end{cases}
\end{align*}
For any subword $v$ of $w$, we abuse the notation and let $\delta^+(v)$ (resp. $\delta^-(v)$) denote the image of applying $\delta^+$ (resp. $\delta^-$) on each letter of $v$.

Suppose $w=w_1\cdots w_n\in \RG(n)$ has the following decomposition, wherein $1$ is the rightmost $1$ in $w$, $k\ge 1$ is the greatest letter in the prefix of $w$ ending at $1$, and $k+1$, if any, is the leftmost $k+1$ in $w$. It could also be the case that $\bk(w)=k$, so no $k+1$ exists.
\begin{align*}
& w=u a_0a_1\cdots a_r 1 b_1\cdots b_s c_1\cdots c_t (k+1)v, \text{ where}\\
& a_0<a_1\ge a_2\ge \cdots \ge a_r\ge 1, \; 1<b_1<b_2<\cdots<b_s\ge c_1.
\end{align*}
Note that $u$ and $v$ are the prefix and suffix of $w$, respectively, and there are no restrictions on the relative orders among $c_1, c_2,\ldots, c_t$. We define our bijection $\xi$ according to the following three cases.
\begin{itemize}
	\item[I.] If $a_0$ does not exist, i.e., the prefix of $w$ is completely composed of $1$, say $w=1\cdots 1\hat{w}$, with $\hat{w}$ containing no $1$, then $\xi(w)=1\cdots 1\delta^-(\hat{w})$.
	\item[II.] If $a_0$ does exist, and $a_r\le b_1-2$ or $b_1$ does not exist, then $$\xi(w)=ua_0ka_1\cdots a_r\delta^-(b_1\cdots b_s c_1\cdots c_t(k+1)v).$$
	\item[III.] If both $a_0$ and $b_1$ exist, and $a_r\ge b_1-1$, then $$\xi(w)=ua_0a_1\cdots a_r\delta^-(b_1\cdots b_s)k\delta^-(c_1 \cdots c_t(k+1)v).$$
\end{itemize}
It can be easily checked that in all three cases,
\begin{align}\label{L-Asc}
\LL(w) = \LL(\xi(w)),\text{ and } \Asc(w) = \Asc(\xi(w)).
\end{align}

\begin{example}
In the following two examples, we calculate two triples of statistics that turn out to be the same, via $\xi$ described above.
\begin{align*}
w &= 12134243221435322, \quad \xi(w) = 12134244322325211,\\
(\LL,\:& \Asc, \wls)\:w =(\LL, \Asc, \ls)\:\xi(w)=(\{1,2,4,5,14\},\{1,3,4,6,11,13\},25);\\
v &= 1232111224254, \quad \xi(v) = 1232111314153,\\
(\LL,\:& \Asc, \wls)\:v =(\LL, \Asc, \ls)\:\xi(v)=(\{1,2,3,10,12\},\{1,2,7,9,11\},15).
\end{align*}
\end{example}
\begin{Def}
Given any word $w$, we call a letter $x$ in $w$ \emph{admissible} if
\begin{itemize}
	\item $w$ contains only one $x$, or
	\item there exists at least one $x+1$ between the first $x$ and the second $x$.
\end{itemize}
\end{Def}
Now for any $w\in\RG(n)$, we find its preimage under $\xi$ by taking the following steps.
\begin{enumerate}
	\item Find the smallest $k$, such that for each $k\le l\le \bk(w)$, $l$ is admissible in $w$.
	\item If $k=1$, then $\xi^{-1}(w)=\delta^+(w)$.
	\item If $k>1$, then locate the leftmost $k$ and the rightmost $k-1$ to the left of this $k$, say $w=ta(k-1)bukv$, where $a,b$ are letters and $t,u,v$ are subwords. Note that since $k-1$ is not admissible by our choice of $k$, the subword $ta$ must contain at least one $k-1$, thus being nonempty.
	\item If $a=k-1$ or $a<b$, suppose $(k-1)buk=x_1x_2\cdots x_m$. Then find the smallest $1\le j < m$, such that $x_j<x_{j+1}$, and take $$\xi^{-1}(w)=tax_2\cdots x_j 1\delta^+(x_{j+1}\cdots x_m v).$$
	\item If $k-1 > a\ge b$, suppose $ta(k-1)=y_1\cdots y_n$. Then find the largest $1\le j<n$, such that $y_j\ge y_{j+1}$, and take $$\xi^{-1}(w)=y_1\cdots y_j 1 \delta^+(y_{j+1}\cdots y_{n-1}bukv).$$
\end{enumerate}
\begin{proof}[Proof of Theorem~\ref{main-bijection}]
The constructions of $\xi$ and $\xi^{-1}$ guarantees that they are indeed inverse to each other. In view of \eqref{L-Asc}, all it remains is to show that
\begin{align}\label{vls-ls}
\wls(w)=\ls(\xi(w)).
\end{align}

We detail case II here and the other two cases can be verified analogously. Recall that in case II,
\begin{align*}
w &=\overbrace{u a_0a_1\cdots a_r}^{\text{zone I}} 1 \overbrace{b_1\cdots b_s c_1\cdots c_t (k+1)v}^{\text{zone II}},\\
\xi(w) &=ua_0ka_1\cdots a_r\delta^-(b_1\cdots b_s c_1\cdots c_t(k+1)v).
\end{align*}
For any $j~(2\le j \le k)$, if the rightmost occurrence of $j$ is in zone I, say $j=w_s\in u a_0a_1\cdots a_r$, then according to the definition of $\wls$, this $w_s$ contributes one more to $\wls(w)$ than to $\ls(\xi(w))$. If the rightmost occurrence of $j$ is in zone II, say $j=w_t\in b_1\cdots b_s c_1\cdots c_t (k+1)v$, then due to operator $\delta^-$, $w_s$ contributes one more to $\wls(w)$ than to $\ls(\xi(w))$. The rest of the occurrences of $j$ contribute equally in $\wls(w)$ and $\ls(\xi(w))$. On the other hand, the newly inserted $k$ in $\xi(w)$ (after $a_0$) now adds $k-1$ to $\ls(\xi(w))$.

For the remaining $j~(k+1\le j\le\bk(w))$, if any, they should all appear in zone II. If this $j$ appears only once, then it contributes $j-1$ to both $\wls(w)$ and $\ls(\xi(w))$. Otherwise we just consider the leftmost occurrence, the middle occurrences, and the rightmost occurrence of $j$, respectively, to realize that collectively they contribute the same amount to both $\wls(w)$ and $\ls(\xi(w))$. Finally, all the occurrences of $1$ contribute nothing either to $\wls(w)$ or to $\ls(\xi(w))$.

We have established \eqref{vls-ls} for case II and the proof is now completed.
\end{proof}

\section{A bijective proof of Theorem~\ref{main-bijection-r}}\label{sec:bijection}
In this section, we will give a bijective proof of Theorem
~\ref{main-bijection-r}. More precisely, we will establish a stronger version as follows.

\begin{theorem}\label{strong-main-bijection-r}
Statistics $(\LL, \Asc, \brs)$ and
$(\LL, \Asc, \bwrs)$ have the same joint distribution on $\RG(n)$ for all $n \geq  1$.
\end{theorem}

We now give an alternative characterization of RGF that will be useful later.

\begin{proposition}\label{char-RGF}
Given a word $w\in [k]^n$, suppose $\LL(w)=\{i_1,i_2,\ldots,i_j\}$ with $i_1<i_2<\cdots<i_j$. Then $w$ is an RGF if and only if $j=\bk(w)$ and $w_{i_1}=1,w_{i_2}=2,\ldots,w_{i_j}=j$.
\end{proposition}

To give a proof of Theorem \ref{strong-main-bijection-r}, we first present a bijection $\varphi$ on $\RG(n)$ which maps
$(\LL, \bwrs)$ to $(\LL, \brs)$, although it does not preserve the set-valued statistic $\Asc$.

Given $w=w_1 w_2 \cdots w_n \in \RG(n)$, we construct $v=v_1 v_2 \cdots v_n:=\varphi(w)$ with respect to the following three cases.

\begin{itemize}
  \item[I.] If $i \in \LL(w)$, then let $v_i = w_i$.
  \item[II.] If $i \in \RR(w) \setminus\LL(w)$ and there exists $j <i$ such that $w_j=w_i-1$, then let $v_i = w_i -1$.
 \item[III.] If $i \in [n]$ does not belong to the above two cases, then let  $w_{n+1}=\bk(w)+1$ and $\suf_i(w)=\{w_{i+1}, \ldots, w_n, w_{n+1}\}$. We take  $v_i=d_i -1$, where
        $d_i=\min\{a \in \suf_i (w):~ a > w_i\}$. Note that $v_i\ge w_i$.
\end{itemize}

To show that $\varphi$ is a bijection,
 for any $v=v_1v_2\cdots v_n \in \RG(n)$, we construct $w=w_1w_2\cdots w_n:=\varphi^{-1}(v)$ according to the following three cases.
\begin{itemize}
  \item[1.] If $i \in \LL(v)$, then let $w_i=v_i$.
  \item[2.] If $i \in \RR(v)\setminus\LL(v)$ and there exists $j <i $ such that $v_j=v_i+1$, then let $w_i=v_i +1$.
  \item[3.]  If $i \in [n]$ does not belong to the above two cases, then let $v_{n+1}=0$ and $\suf_i(v)=\{v_{i+1}, \ldots, v_n, v_{n+1}\}$. We take $w_i=t_i +1$, where
        $t_i=\max\{a \in \suf_i(v):~ a < v_i\}$. Note that $w_i\le v_i$.
\end{itemize}

Suppose $v=\varphi(w)$ for a given $w\in\RG(n)$. It can be easily checked that $i$ is case I if and only if $i$ is case 1, hence $\LL(w)=\LL(v)$ and $v\in\RG(n)$ by Proposition~\ref{char-RGF}. Therefore $\varphi$ is well-defined. Similarly, one checks that $i$ is case II if and only if $i$ is case 2. Since the three cases I, II and III (resp. the three cases 1, 2 and 3) are mutually exclusive, we see that $i$ is case III if and only if $i$ is case 3. By our construction of $\varphi$ and $\varphi^{-1}$, they are clearly inverse to each other casewisely. Consequently $\varphi$ is bijective. In what follows, we wish to show that $\bwrs(w)=\brs(v)$.
First, we give an alternative definition of the statistic $\wrs_i$,
which can be easily checked.
\begin{align}\label{alter-vrs}
  \wrs_i(w) &=\#\{j \in \RR(w):~j>i,~ w_i \geq w_j >1 \}.
\end{align}

\begin{Def}
Given $w=w_1 w_2 \cdots w_n \in \RG(n)$, for $1 \leq i <j \leq n$, we call $(w_i,w_j)$
a \emph{$\wrs$-pair} if  $j \in \RR(w)$, $w_i \geq w_j $ and $ w_j \neq 1$. We call $(w_i,w_j)$ an \emph{$\rs$-pair} if
 $j \in \RR(w)$ and $w_i > w_j $.
\end{Def}

\begin{lemma}\label{vrs-rs-pair}
For $w\in\RG(n)$, let $v=\varphi(w)$. For $1 \leq i <j \leq n$,   $(w_i, w_j)$ is
a $\wrs$-pair of $w$ if and only if $(v_i,v_j)$ is an $\rs$-pair of
$v$.
\end{lemma}

\begin{proof}
  Suppose that $(w_i, w_j)$ is a $\wrs$-pair of $w$, we proceed to show that $(v_i , v_j)$ is an $\rs$-pair of $v$.
  By definition, $j \in \RR(w)$ and $w_j \neq 1$. Furthermore,  since $w_i \geq w_j$ and $w\in\RG(n)$, we have $j \notin \LL(w)$. Thus, $j$
  belongs to case II in the construction of $\varphi$.
  Hence $v_j=w_j-1$.
  By the fact that $j \in \RR(w)$, we have $v_k \neq w_j -1$
 for  $j+1 \leq k \leq n$. It follows that $j \in \RR(v)$.

    If $i$ is case I, then $v_i=w_i$, which implies that $v_i>v_j$. If $i$ belongs to case II, then
    $w_i> w_j$ and $v_i=w_i-1$. It follows that
    $v_i >v_j$. If $i$ belongs to case III, $v_i \geq w_i > w_j-1=v_j$.
    Thus for all  the three cases above, we see $(v_i , v_j)$ is an
    $\rs$-pair of $v$.

    Conversely, suppose that $(v_i,v_j)$ is an $\rs$-pair of $v$, we need to
    show that $(w_i,w_j)$ is a $\wrs$-pair of $w$. Clearly, $j \in \RR(v)\setminus\LL(v)$. If  $v_k \neq v_j +1$ for any $1 \leq k <j$, then $v_k \le v_j$ for every $1 \leq k <j$. In particular $v_i \leq v_j$, which contradicts with the fact that $(v_i,v_j)$ is an $\rs$-pair.
    Hence, there exists $k <j $ such that $v_k=v_j+1$. This implies that $j$ is case $2$ in the construction of $\varphi^{-1}$.
    Then $w_j=v_j+1> 1$ and $j\in \RR(w)$.

    If $i$ is case $1$, then
    $w_i=v_i\geq v_j+1=w_j$. If $i$ is case $2$, then
    $w_i=v_i+1> v_j+1=w_j$. If $i$ is case $3$, then
    $w_i \geq v_j+1=w_j$. Thus for all the three cases above, we
    see that $(w_i,w_j)$ is a $\wrs$-pair of $w$. The proof is now completed.
\end{proof}

Clearly, each $\wrs$-pair of $w$  with the index of the first element being $i$ contributes $1$ to
$\wrs_i(w)$, while each $\rs$-pair of $w$  with the index of the first element being $i$ contributes $1$ to
$\rs_i(w)$. Hence, by Lemma \ref{vrs-rs-pair} we deduce that
$\bwrs(w)=\brs(v)$ for $v=\varphi(w)$, as desired.

\begin{example}
In the following two examples, we calculate the images of $w$ and $u$ under $\varphi$,
as well as the corresponding statistics $\LL$,$\bwrs$, $\brs$ and $\Asc$.
\begin{align*}
&w= 12131435564341474, \quad  \varphi(w) = 12232435466263673,\\
(\LL,\: \bwrs)\:w =(\LL, \brs)\:&\varphi(w)=(\{1,2,4,6,8,10,16\},(0,0,0,1,0,2,1,3,2,2,2,0,1,0,1,1,0)),\\
&\Asc(w) = \Asc(\varphi(w))=\{1,3,5,7,9,12,14,15\};\\
&u = 123454333673573, \quad \varphi(u) = 123453444674462,\\
(\LL,\:  \bwrs)\:u =(\LL, \brs)\:&\varphi(u)=(\{1,2,3,4,5,10,11\},\{0,0,1,2,3,1,1,1,1,2,3,1,1,1,0\}),\\
\Asc(u)=&\{1,2,3,4,9,10,12,13\}, \quad \Asc(\varphi(u))=\{1,2,3,4,6,9,10,13\}.
\end{align*}
\end{example}

From the examples above, we see that $\varphi$ does not
always keep the statistic $\Asc$. To fix this, we are going to compose $\varphi$ with another bijection $\gamma$. The following observation is the first step. It tells us that we only need to worry about two particular cases.

\begin{lemma}\label{notkeepasc}
Given $w=w_1\cdots w_n\in\RG(n)$, let $v=\varphi(w)$. If neither of the following two
cases happens, then $\Asc(w)=\Asc(v)$. Recall that $\suf_i(w)=\{w_{i+1},\ldots,w_n,w_{n+1}\}$, where $w_{n+1}=\bk(w)+1$.
\begin{itemize}
  \item[1)] $i$ is case II, $i+1$ is case III, $w_i>w_{i+1}$, and
  $\min\{a \in \suf_{i+1}(w): a >w_{i+1}\}>w_i$;
  \item[2)] $i$ is case III, $i+1$ is case II, $w_i<w_{i+1}$, and
  $\min\{a \in \suf_{i}(w): a >w_{i}\}=w_{i+1}$.
\end{itemize}
Otherwise, if $(w_i,w_{i+1})$ satisfies condition 1), then $i\notin\Asc(w)$ and $i\in \Asc(v)$. If $(w_i,w_{i+1})$ satisfies condition 2), then $i\in\Asc(w)$ and $i\notin \Asc(v)$.
\end{lemma}

\begin{proof}
For the first statement, it will suffice to show that under the above conditions, if $i \in \Asc(w)$, then $i \in \Asc(v)$, and if $i \notin \Asc(w)$, then $i \notin \Asc(v)$. There are in total nine cases to consider, according to which of the three cases $i$ and $i+1$ belong to respectively. We include here the proofs of two cases and omit the analogous proofs of the remaining ones.

We first consider the case when both $i$ and $i+1$ are case III.
If $i \in \Asc(w)$, then
\[
\min\{a \in \suf_i(w): a> w_i\} \leq w_{i+1} < \min\{a \in \suf_{i+1}(w): a> w_{i+1}\}.
\]
It follows that $v_i < v_{i+1}$, that is, $i \in \Asc(v)$.

If $i \notin \Asc(w)$, we have $w_i \ge w_{i+1}$, then
\[\min\{a \in \suf_{i}(w): a> w_i\} \geq \min\{a \in \suf_{i+1}(w): a> w_{i+1}\}.\] It follows that $v_i \geq v_{i+1}$, that is, $i \notin \Asc(v)$. This
completes the proof of this case.

Next we consider a case that actually uses condition 1). Namely, the case when $i$ is case II and $i+1$ is case III. If $i\in \Asc(w)$, then $w_i<w_{i+1}$, so $v_i=w_i-1<w_{i+1}\le v_{i+1}$, that is, $i\in \Asc(v)$.

If $i\notin \Asc(w)$, then $w_i\ge w_{i+1}$. But $w_i=w_{i+1}$ is absurd since $i\in \RR(w)$. So we must have $w_i>w_{i+1}$. Now we apply the restriction in 1) and $i\in \RR(w)$ again, to see that $$\min\{a \in \suf_{i+1}(w): a >w_{i+1}\}<w_i.$$ Thus $v_i=w_i-1 > v_{i+1}$, that is, $i\notin \Asc(v)$. So we have proved $\Asc(w)=\Asc(v)$.

The statements after ``otherwise'' follow directly from the definition of $\varphi$. The proof is now completed.
\end{proof}

\begin{Def}\label{badpair}
For $w=w_1\cdots w_n\in\RG(n)$, if $(w_i,w_{i+1})$ is a pair that satisfies condition 1) in Lemma~\ref{notkeepasc}, then we call it a \emph{bad pair of type 1}. If $(w_i,w_{i+1})$ is a pair that satisfies condition 2) in Lemma~\ref{notkeepasc}, then we call it a \emph{bad pair of type 2}. For any word $w\in [k]^n$, we define the \emph{swap operator} $\swap$ by letting $\swap_i(w)=w_1\cdots w_{i-1}w_{i+1}w_iw_{i+2}\cdots w_n$, for a given $1\le i<n$.
\end{Def}

The main idea in our construction of $\gamma$ is repeatedly applying $\swap_i$ for each bad pair $(w_i,w_{i+1})$. We need the following lemma to understand what happens to the statistics after the swapping.

\begin{lemma}\label{stat-swap}
Let $w=w_1\cdots w_n\in\RG(n)$. If $(w_i,w_{i+1})$ is a bad pair of type 1 (resp. of type 2) in $w$, then $(w_{i+1},w_i)$ is a bad pair of type 2 (resp. of type 1) in $\swap_i(w)$. Moreover, we have $\swap_i(w)\in\RG(n)$, $\LL(w)=\LL(\swap_i(w))$, $\bwrs(w)=\bwrs(\swap_i(w))$ and $i\in\Asc(\swap_i(w))$ (resp. $i\notin\Asc(\swap_i(w))$).
\end{lemma}

\begin{proof}
The first statement is a direct result of the complemental restrictions for type 1 and type 2 bad pairs (see conditions 1) and 2) in Lemma~\ref{notkeepasc}). To show the remaining statements, first note that the swap operator preserves the position of each leftmost occurrence of $w$, thus $\swap_i(w)\in \RG(n)$ by Proposition~\ref{char-RGF}, and in particular we have $\LL(w)=\LL(\swap_i(w))$. Using the alternative definition given by \eqref{alter-vrs}, it should also be clear that for $j\notin\{i,i+1\}$, we have $\wrs_j(w)=\wrs_j(\swap_i(w))$. The last statement involving $\Asc$ is also obvious. All it remains is to show that $\wrs_i(w)=\wrs_i(\swap_i(w))$ and $\wrs_{i+1}(w)=\wrs_{i+1}(\swap_i(w))$. Actually a stronger relation holds, we claim that
$$\wrs_i(w)=\wrs_i(\swap_i(w))=\wrs_{i+1}(w)=\wrs_{i+1}(\swap_i(w)).$$

We only prove this claim for $(w_i,w_{i+1})$ being a bad pair of type 1. The proof goes similarly when $(w_i,w_{i+1})$ is a bad pair of type 2 and thus is omitted.

If $i+1=\pl(w)$, then according to the definitions of type 1 bad pair and $\wrs$, we have
$$\wrs_i(w)=\wrs_i(\swap_i(w))=\wrs_{i+1}(w)=\wrs_{i+1}(\swap_i(w))=0.$$
Otherwise $w_{i+1}>1$ and since $i+1$ is case III, there exists some $k>i+1$ such that $w_k=w_{i+1}$. This implies that
\begin{align*}
\wrs_i(w)&=\#\{j \in \RR(w): j>i,~ w_i \geq w_j> 1 \}\\
&=\#\{j \in \RR(w):j>i+1,~ w_i \geq w_j> 1 \}=\wrs_{i+1}(\swap_i(w)).
\end{align*}
It is also clear that $\wrs_{i+1}(w)=\wrs_i(\swap_i(w))$ since $w_i>w_{i+1}$. We use the condition $$\min\{a \in \suf_{i+1}(w): a >w_{i+1}\}>w_i$$ again to deduce that $\wrs_i(w)=\wrs_{i+1}(w)$. The claim is now proved.
\end{proof}

Now we are ready to define $\gamma$ as a map from $\RG(n)$ to $\RG(n)$, via the following \emph{sweep algorithm}.

\begin{framed}
\begin{center}
\bf Sweep Algorithm
\end{center}
\noindent Input $w\in\RG(n)$, set $u=w$.\hfill

\noindent For $i=1$ to $n-1$, $i++$,\hfill

if $(u_i,u_{i+1})$ is a bad pair, then set $u=\swap_i(u)$, \hfill

\quad otherwise go on.\hfill

\noindent Output $\gamma(w):=u$.\hfill
\end{framed}

Since $\swap_i(w)\in\RG(n)$ as proved in Lemma~\ref{stat-swap}, we see that $\gamma:\RG(n)\rightarrow\RG(n)$ is well-defined. Its bijectivity is evident from the following \emph{reverse sweep algorithm}.

\begin{framed}
\begin{center}
\bf Reverse Sweep Algorithm
\end{center}
\noindent Input $w\in\RG(n)$, set $u=w$.\hfill

\noindent For $i=n-1$ to $1$, $i--$,\hfill

 if $(u_i,u_{i+1})$ is a bad pair, then set $u=\swap_i(u)$, \hfill

\quad otherwise go on.\hfill

\noindent Output $\gamma^{-1}(w):=u$.\hfill
\end{framed}

\begin{example} Let $w=12345433673573$, we illustrate the sweep algorithm that outputs $\gamma(w)$ step-by-step. The swapped pairs have been underlined for each step.
\begin{align*}
w &=12345433673573\xrightarrow[\text{type 1}]{\swap_6}12345\underline{34}3673573\xrightarrow[\text{type 1}]{\swap_7}123453\underline{34}673573\\
&\xrightarrow[\text{type 2}]{\swap_{11}}1234533467\underline{53}73\xrightarrow[\text{type 2}]{\swap_{12}}12345334675\underline{73}3=\gamma(w).
\end{align*}
\end{example}

Given $w\in\RG(n)$, we apply the sweep algorithm on $w$. For $1\le i<n$, we define $w^{(i)}$ to be the word we get after the $i$-th round in the algorithm. More precisely, if $(w_1,w_2)$ is a bad pair in $w$, then let $w^{(1)}=\swap_1(w)$, otherwise let $w^{(1)}=w$. Next, if $(w^{(1)}_2,w^{(1)}_3)$ is a bad pair in $w^{(1)}$, then let $w^{(2)}=\swap_2(w^{(1)})$, otherwise let $w^{(2)}=w^{(1)}$, so on and so forth. In particular, we see that $w^{(n-1)}=\gamma(w)$. The next lemma is stronger than what we need, but is easier to prove by induction when it is stated this way.

\begin{lemma}\label{lemma:Asc}
For $1\le i<n$, $\Asc(w)\cap[i]=\Asc(\varphi(w^{(i)}))\cap[i]$.
\end{lemma}
\begin{proof}
We use induction on $i$. For $i=1$, note that $(w_1,w_2)$ cannot be a bad pair since 1 is always case I, then $w^{(1)}=w$ and the identity follows from Lemma~\ref{notkeepasc}. Now suppose the identity has been verified for $i=1,2,\ldots,j$. The case of $(w^{(j)}_{j+1},w^{(j)}_{j+2})$ not being a bad pair is again taken care of by Lemma~\ref{notkeepasc}. So we assume that $(w^{(j)}_{j+1},w^{(j)}_{j+2})$ is a bad pair, say of type 1 (the proof of type 2 is quite similar), in $w^{(j)}$, then $w^{(j)}_{j+1}>w^{(j)}_{j+2}$, $j+1\notin\Asc(w^{(j)})$. Clearly
\begin{align}\label{j-1}
\Asc(w)\cap[j-1]=\Asc(\varphi(w^{(j)}))\cap[j-1]=\Asc(\varphi(w^{(j+1)}))\cap[j-1].
\end{align}
For each word involved, we now need to take a closer look at the $j$-th, $j+1$-th and $j+2$-th letters. For brevity, let $u=\varphi(w^{(j)})$ and $v=\varphi(w^{(j+1)})$. Note that $w^{(j)}_j=w^{(j+1)}_j$, $u_j=v_j$ and $u_{j+1}=w^{(j)}_{j+1}-1=v_{j+1}$. This means that $u_j<u_{j+1}$ if and only if $v_j<v_{j+1}$, hence together with \eqref{j-1} we have
\begin{align}\label{j}
\Asc(w)\cap[j]=\Asc(\varphi(w^{(j)}))\cap[j]=\Asc(\varphi(w^{(j+1)}))\cap[j].
\end{align}
Finally, since $(w^{(j)}_{j+1},w^{(j)}_{j+2})$ is a bad pair of type 1, we see $j+1\notin\Asc(w^{(j)})$ and $j+1\in\Asc(u)$. By Lemma~\ref{stat-swap}, $(w^{(j+1)}_{j+1},w^{(j+1)}_{j+2})$ is a bad pair of type 2, hence $j+1\in\Asc(w^{(j+1)})$ and $j+1\notin\Asc(v)$. Moreover, we have
\begin{align}\label{j+1}
w_{j+2}=w^{(1)}_{j+2}=\cdots=w^{(j)}_{j+2},\;w_{j+1}=w^{(1)}_{j+1}=\cdots=w^{(j-1)}_{j+1}.
\end{align}
If $(w^{(j-1)}_j,w^{(j-1)}_{j+1})$ is not a bad pair, then $w^{(j-1)}=w^{(j)}$, in particular $w^{(j-1)}_{j+1}=w^{(j)}_{j+1}$. Together with \eqref{j+1} we have $j+1\notin \Asc(w)$. If $(w^{(j-1)}_j,w^{(j-1)}_{j+1})$ is a bad pair, then it must be of type 1, since $j+1$ is case II in $w^{(j)}$ and $w^{(j-1)}_j=w^{(j)}_{j+1}$. The condition of type 1 bad pair forces $w_{j+1}=w^{(j-1)}_{j+1}\ge w^{(j-1)}_{j+2}=w_{j+2}$, so once again we have $j+1\notin\Asc(w)$. In summary, we have $j+1\notin\Asc(w)$ and $j+1\notin\Asc(v)$. Combining this with \eqref{j} we arrive at
$$\Asc(w)\cap[j+1]=\Asc(\varphi(w^{(j+1)}))\cap[j+1].$$
The proof is now completed by induction.
\end{proof}

We find ourselves in the position to prove Theorem \ref{strong-main-bijection-r}, which implies Theorem \ref{main-bijection-r}.
\begin{proof}[Proof of Theorem~\ref{strong-main-bijection-r}]
Let $\zeta=\varphi  \circ  \gamma$. Notice that both $\varphi$ and $\gamma$ are bijections on $\RG(n)$. Thus, $\zeta$ is a bijection on
$\RG(n)$.

By Lemmas~\ref{vrs-rs-pair} and \ref{stat-swap}, we see that $\gamma$ preserves $(\LL,\bwrs)$, and $\varphi$ maps $(\LL, \bwrs)$ to $(\LL, \brs)$. It remains to show that $\Asc(w)=\Asc(\zeta(w))$. Since $\Asc(w)\subseteq [n-1]$ for any word $w$ of length $n$, we see $\Asc(w)=\Asc(\zeta(w))$ is simply the $i=n-1$ case of Lemma~\ref{lemma:Asc}.

\end{proof}
\section{\texorpdfstring{A refined relation between $\BAST$ and $\STAT$}{}}\label{sec:stat-bast}
Among the new Mahonian statistics introduced by Babson and Steingr\'imsson \cite{BS}, the following one
\begin{align}\label{def:stat}
\STAT &=\underline{21}+\underline{13}2+\underline{21}3+\underline{32}1
\end{align}
has attracted much attention (see for example \cite{FZ,Bur,KV,CL,FHV}). The equidistribution results we have seen so far in this paper are between $\BAST=\STAT^{\rc}$ and $\MAJ$ (see Corollaries~\ref{main-Sn} and \ref{main-bijectionper}), over $\SS_n(1\underline{32})$. The natural quest to establish comparable results directly between $\STAT$ and $\MAJ$ led us to explore further the relation between $\STAT$ and $\BAST$, and find the following theorem, which does not follow from applying reverse complement. Given a permutation $\pi=\pi_1\pi_2\cdots\pi_n$, we define two statistics, the first letter $\F(\pi):=\pi_1$, and the ending letter $\L(\pi):=\pi_n$.
\begin{theorem}\label{thm:stat-bast}
Statistics $(\F, \L, \des, \Id, \STAT, \BAST)$ and
$(\F, \L, \des, \Id, \BAST, \STAT)$ have the same joint distribution on $\SS_n$ for all $n \geq 1$.
\end{theorem}

To prove Theorem~\ref{thm:stat-bast}, we construct an involution $\psi$, which builds on an involution $\phi$ introduced in \cite{FHV} to give a proof of the following theorem. Let $R(w)$ denote the rearrangement class of a given word $w$. For instance, if $w=1314$, then $$R(w)=\{1134, 1143, 1314, 1341, 1413, 1431, 3114, 3141, 3411, 4113, 4131, 4311\}.$$
\begin{theorem}[Theorem 1.3 in \cite{FHV}]
\label{thm:FHV}
There exists an involution $\phi=\phi_{R(w)}$ for each rearrangement class $R(w)$, such that $$(\F, \des, \Id, \MAJ, \STAT)v=(\F, \des, \Id, \STAT, \MAJ)\phi(v),$$ for any $v\in R(w)$.
\end{theorem}
The following operators $\gamma^{+}$ and $\gamma^{-}$ are crucial in our construction of $\psi$.
\begin{Def}
Given a word $w=w_1w_2\cdots w_n$ of length $n$, we define
\begin{align*}
\gamma^{+}(w) &=u_1u_2\cdots u_n, \text{ where } u_j=\begin{cases}
w_j+n & \text{if }w_j<w_n,\\
w_j & \text{otherwise}.
\end{cases}\\
\gamma^{-}(w) &=u_1u_2\cdots u_n, \text{ where } u_j=\begin{cases}
w_j-n & \text{if }w_j>n,\\
w_j & \text{otherwise}.
\end{cases}
\end{align*}
\end{Def}
Clearly we have $\gamma^-\circ\gamma^+(\pi)=\pi$, for any $\pi\in\SS_n$. We also need \emph{the $i$-th prefix of $w=w_1w_2\cdots w_n$}, denoted as $\pre_i(w):=w_1w_2\cdots w_i$, for $1\le i\le n$. Now we can define $\psi$.
\begin{Def}
For the single permutation of length $1$, we let $\psi(1)=1$. For $n\ge 2$, given a permutation $\pi=\pi_1\pi_2\cdots\pi_n\in\SS_n$, if $\pi_1>\pi_n$, then we let
$$\psi(\pi)=\gamma^-(\phi(\pre_{n-1}(\gamma^+(\pi)))\cdot \pi_n),$$
where $w\cdot u$ denotes the concatenation of two words $w$ and $u$. If we have $\pi_1<\pi_n$, then we let
$$\psi(\pi)=r(\psi(r(\pi))),$$
where $r$ is the reversal.
\end{Def}
The following observation on $\psi$ should be clear from its definition and the properties of $\phi$ revealed by Theorem~\ref{thm:FHV} above. Namely, $\phi$ is an involution on each rearrangement class of words, which preserves the first letter statistic $\F$.

\begin{lemma}\label{psi:F-L}
For $n\ge 1$, the map $\psi$ is an involution on $\SS_n$, such that for any $\pi\in\SS_n$,
$$\F(\pi)=\F(\psi(\pi)),\; \L(\pi)=\L(\psi(\pi)).$$
\end{lemma}

To show the equidistribution of the remaining statistics, we need three more lemmas.

\begin{lemma}\label{psi:des}
For any $\pi\in\SS_n$ with $\pi_1>\pi_n$, we have
\begin{align}\label{gamma:des}
\des(\pi)=\des(\gamma^{+}(\pi)), \text{ and}
\end{align}
\begin{align}\label{gamma:MAJ}
\MAJ(\gamma^+(\pi))=\MAJ(\pi)+\pi_n-1.
\end{align}
Consequently, for any $\pi\in\SS_n$, we have $\des(\pi)=\des(\psi(\pi))$.
\end{lemma}
\begin{proof}
First we assume that $\pi_1>\pi_n$. To show that $\des(\pi)=\des(\gamma^{+}(\pi))$, we view $\pre_{n-1}(\pi)$ as the concatenation of alternating large blocks $B$'s and small blocks $b$'s, such that letters in $B$'s are all larger than $\pi_n$, while letters in $b$'s are all smaller than $\pi_n$. Since $\pi_1>\pi_n$, the first block must be a large block, we then have the following two cases.
\begin{itemize}
	\item[I.] $\pre_{n-1}(\pi)=B_1b_1B_2b_2\cdots B_kb_k$. Then $\gamma^+(\pi)=B_1\tilde{b}_1B_2\tilde{b}_2\cdots B_k\tilde{b}_k\pi_n$, where each letter in $\tilde{b}$'s has been increased by $n$.
	\item[II.] $\pre_{n-1}(\pi)=B_1b_1B_2b_2\cdots B_k$. Then $\gamma^+(\pi)=B_1\tilde{b}_1B_2\tilde{b}_2\cdots B_k\pi_n$.
\end{itemize}
The descents in $\pi$ come in two types, namely those with both letters coming from the same block, called \emph{in-block descents}, and the remaining ones called \emph{inter-block descents}. In both cases I and II, the number and the positions of the in-block descents are preserved under the operation $\gamma^{+}$. For inter-block descents, suppose $\pi$ is in case I. Then the ending letter of $B_1$ and the starting letter of $b_1$ form a descent of $\pi$. The same holds true for $B_2$ and $b_2,\cdots$, for $B_k$ and $b_k$. The total contribution to $\des(\pi)$ is $k$. On the other hand, we see the ending letter of $\tilde{b}_1$ and the starting letter of $B_2$ form a descent of $\gamma^{+}(\pi)$. The same holds true for $\tilde{b}_2$ and $B_3,\ldots$, for $\tilde{b}_k$ and $\pi_n$. Therefore the total contribution to $\des(\gamma^+(\pi))$ is also $k$, but to $\MAJ(\gamma^+(\pi))$ is increased by $|\tilde{b_1}|+|\tilde{b_2}|+\cdots+|\tilde{b_k}|=\pi_n-1$. A similar analysis applies for case II. So we see indeed $\des(\pi)=\des(\gamma^+(\pi))$ and \eqref{gamma:MAJ} holds true.

Still assuming $\pi_1>\pi_n$, now we show that $\des(\pi)=\des(\psi(\pi))$. The following claim is crucial and not hard to check.
\begin{align*}
\text{If $q\in\gamma^+(\SS_n)$, then } \phi(\pre_{n-1}(q))\cdot q_n\in\gamma^+(\SS_n).
\end{align*}
Let $p=\gamma^+(\pi)$, note that $p_n=\pi_n$ by definition. By the above claim, there exists some $\sigma\in\SS_n$, such that $\phi(\pre_{n-1}(p))\cdot p_n=\gamma^+(\sigma)$, hence $\psi(\pi)=\gamma^-(\gamma^+(\sigma))=\sigma$. Therefore by Lemma~\ref{psi:F-L} we have $\sigma_1=\pi_1>\pi_n=\sigma_n$, now we see
\begin{align*}
\des(\psi(\pi)) &= \des(\sigma)=\des(\gamma^+(\sigma))=\des(\phi(\pre_{n-1}(p))\cdot p_n)\\
&=\des(\phi(\pre_{n-1}(p)))+1=\des(\pre_{n-1}(p))+1\\
&=\des(p)=\des(\gamma^+(\pi))=\des(\pi).
\end{align*}

Finally, for those $\pi\in\SS_n$ with $\pi_1<\pi_n$, simply note that $\des(r(\sigma))=n-1-\des(\sigma)$ for any $\sigma\in\SS_n$.
\end{proof}

\begin{lemma}\label{gamma:stat}
For any $\pi\in\SS_n$ with $\pi_1>\pi_n$, we have $\STAT(\pi)=\STAT(\gamma^+(\pi))$.
\end{lemma}
\begin{proof}
In view of the definition \eqref{def:stat} and equation \eqref{gamma:des}, it suffices to show that
$$\underline{13}2(\pi)+\underline{21}3(\pi)+\underline{32}1(\pi)=\underline{13}2(\gamma^+(\pi))+\underline{21}3(\gamma^+(\pi))+\underline{32}1(\gamma^+(\pi)).$$

Let $p=p_1p_2\cdots p_n=\gamma^+(\pi)$. If we examine the contributions to all three statistics above according to the underlined pair, i.e., we fix some $1<i< n$, and let $j=i+1,i+2,\ldots,n$, then sum up the contributions from the triple $\pi_{i-1}\pi_i\pi_j$ (resp. $p_{i-1}p_ip_j$) to each of $\underline{13}2,\underline{21}3$ and $\underline{32}1$. We actually observe a stronger result that pairwisely the contributions are equal. We elaborate on one case below, and leave the details of the remaining cases to the interested readers.

As in the proof of Lemma~\ref{psi:des}, we decompose $\pi$ as alternating large blocks and small blocks. Whether $\pi$ is in case I or case II does not matter here. Now suppose $\pi_{i-1}$ is the last letter in $B_m$ and $\pi_i$ is the first letter in $b_m$, then $\pi_{i-1}>\pi_i$, we discuss the value of $\pi_j$ by the following four cases.
\begin{enumerate}[i]
	\item $\pi_j>\pi_{i-1}$ so $\pi_j$ must be contained in a large block. We see $\pi_{i-1}\pi_i\pi_j$ contributes $1$ to $\underline{21}3(\pi)$, while $p_{i-1}p_ip_j=\pi_{i-1}(\pi_i+n)\pi_j$ adds $1$ to $\underline{13}2(p)$.
	\item $\pi_i<\pi_j<\pi_{i-1}$ with $\pi_j$ contained in a large block or $\pi_j=\pi_n$. Then $\pi_{i-1}\pi_i\pi_j$ contributes $0$ to the sum $\underline{13}2(\pi)+\underline{21}3(\pi)+\underline{32}1(\pi)$, while $p_{i-1}p_ip_j=\pi_{i-1}(\pi_i+n)\pi_j$ adds $0$ to the sum $\underline{13}2(p)+\underline{21}3(p)+\underline{32}1(p)$.
	\item $\pi_i<\pi_j<\pi_{i-1}$ with $\pi_j$ contained in a small block. Then $\pi_{i-1}\pi_i\pi_j$ contributes $0$ to the sum $\underline{13}2(\pi)+\underline{21}3(\pi)+\underline{32}1(\pi)$, while $p_{i-1}p_ip_j=\pi_{i-1}(\pi_i+n)(\pi_j+n)$ adds $0$ to the sum $\underline{13}2(p)+\underline{21}3(p)+\underline{32}1(p)$.
	\item $\pi_j<\pi_i$ so $\pi_j$ must be contained in a small block. Then $\pi_{i-1}\pi_i\pi_j$ contributes $1$ to $\underline{32}1(\pi)$, while $p_{i-1}p_ip_j=\pi_{i-1}(\pi_i+n)(\pi_j+n)$ adds $1$ to $\underline{13}2(p)$.
\end{enumerate}

We sum up the contributions from all pairs to arrive at the desired identity and the proof is now completed.
\end{proof}

\begin{lemma}\label{psi:Id}
For any $\pi\in\SS_n$, we have $\Id(\pi)=\Id(\psi(\pi))$.
\end{lemma}
\begin{proof}
We first consider the case when $\pi_1>\pi_n$. In particular, this means $\pi_n<n$, so $\pi_{n}+1$ must appear to the left of $\pi_n$ in both $\pi$ and $\psi(\pi)$. Therefore we have $\pi_n\in\Id(\pi)$ and $\pi_n\in\Id(\psi(\pi))$. And if $\pi_n>1$, similarly $\pi_n-1\not\in\Id(\pi)$ and $\pi_n-1\not\in\Id(\psi(\pi))$. Now for any $\pi_n<m<n$, we see
$$m\in\Id(\pi) \Leftrightarrow m\in\Id(\pre_{n-1}(\gamma^+(\pi))) \Leftrightarrow m\in\Id(\phi(\pre_{n-1}(\gamma^+(\pi)))\cdot\pi_n)\Leftrightarrow m\in\Id(\psi(\pi)).$$
Similarly for any $1\le m<\pi_n-1$, we have
$$m\in\Id(\pi)\Leftrightarrow m+n\in\Id(\pre_{n-1}(\gamma^+(\pi)))\Leftrightarrow m+n\in\Id(\phi(\pre_{n-1}(\gamma^+(\pi)))\cdot\pi_n)\Leftrightarrow m\in\Id(\psi(\pi)).$$
This completes the proof of the $\pi_1>\pi_n$ case. For the case with $\pi_1<\pi_n$, again it suffices to notice that $\Id(r(\sigma))=[n-1]-\Id(\sigma)$.
\end{proof}

Now we are in a position to prove Theorem~\ref{thm:stat-bast}.

\begin{proof}[Proof of Theorem~\ref{thm:stat-bast}]
In view of Lemmas~\ref{psi:F-L}, \ref{psi:des} and \ref{psi:Id}, all it remains is to show that for any $\pi\in\SS_n$, we have $\BAST(\psi(\pi))=\STAT(\pi)$ and $\STAT(\psi(\pi))=\BAST(\pi)$.

We first consider the case with $\pi_1>\pi_n$, and assume that $\phi(\pre_{n-1}(\gamma^{+}(\pi)))\cdot \pi_n=\gamma^+(\sigma)$ for some $\sigma\in\SS_n$. We have
\begin{align*}
\BAST(\psi(\pi))&=\BAST(\gamma^-(\phi(\pre_{n-1}(\gamma^{+}(\pi)))\cdot \pi_n))=\BAST(\gamma^-(\gamma^+(\sigma)))=\BAST(\sigma)\\
&=\MAJ(\sigma)-n+\sigma_n+\des(\sigma) \qquad [\text{by Eq.~\eqref{231-213}}]\\
&=\MAJ(\gamma^+(\sigma))-n+1+\des(\sigma)\qquad [\text{by Eq.~\eqref{gamma:MAJ}}]\\
&=\MAJ(\phi(\pre_{n-1}(\gamma^{+}(\pi)))\cdot \pi_n)-(n-1)+\des(\gamma^+(\sigma))\qquad [\text{by Eq.~\eqref{gamma:des}}]\\
&=\MAJ(\phi(\pre_{n-1}(\gamma^{+}(\pi))))+\des(\phi(\pre_{n-1}(\gamma^{+}(\pi))))+1\\
& \qquad [\text{since $\pi_n$ is the smallest in $\phi(\pre_{n-1}(\gamma^{+}(\pi)))\cdot \pi_n$, must cause a descent.}]\\
&=\STAT(\pre_{n-1}(\gamma^+(\pi)))+\des(\pre_{n-1}(\gamma^+(\pi)))+1\qquad [\text{by Theorem~\ref{thm:FHV}}]\\
&=\STAT(\gamma^+(\pi))\\
&=\STAT(\pi). \qquad [\text{by Lemma~\ref{gamma:stat}}]
\end{align*}
The penultimate equal sign is because $\pi_n$ is the smallest in $\gamma^{+}(\pi)$, thus appending it must cause one more descent and contribute to $\underline{32}1$ for each descent pair in $\pre_{n-1}(\gamma^+(\pi))$.

Recall that $\psi$ is an involution and $\psi(\pi)_1=\pi_1>\pi_n=\psi(\pi)_n$, so we also have
$$\STAT(\psi(\pi))=\BAST(\psi(\psi(\pi)))=\BAST(\pi).$$

The next observation follows directly from the definitions of $\STAT$ amd $\BAST$.
\begin{align*}
\text{For any $\pi\in\SS_n$, }\BAST(\pi)+\STAT(r(\pi))=\BAST(r(\pi))+\STAT(\pi)=\binom{n}{2}.
\end{align*}

Lastly, for the case with $\pi_1<\pi_n$, we have
\begin{align*}
\BAST(\psi(\pi))&=\BAST(r\circ\psi\circ r(\pi))=\binom{n}{2}-\STAT(\psi(r(\pi)))\\
&=\binom{n}{2}-\BAST(r(\pi))=\STAT(\pi).
\end{align*}
Now $\psi$ being an involution again leads to $\STAT(\psi(\pi))=\BAST(\pi)$, which finishes the proof.
\end{proof}

\section{Statistics on ordered set partitions and ordered multiset partitions}\label{sec:stat on OP}
\subsection{Permutations and ordered set partitions}\label{subsec:OP}
We first recall the $q$-Eulerian number
$$
A_q(n,k)=\sum_{p\in\SS_n^k}q^{\MAJ(p)}.
$$
The following relation between $q$-Stirling numbers of the second kind and $q$-Eulerian numbers was first derived by Zeng and Zhang \cite{ZZ}, and is the key component in our proof of Theorem~\ref{main-URG}.
\begin{proposition}
For all $n$ and $k$ with $0\le k \le n$ we have
\begin{align}\label{id:Zeng-Zhang}
[k]_q!\cdot S_q(n,k)=\sum_i q^{k(k-i)}\cdot \qbin{n-i}{k-i}{q}\cdot A_q(n,i-1),
\end{align}
where $$\qbin{n}{k}{q} := \begin{cases}\dfrac{[n]_q!}{[k]_q![n-k]_q!} & \text{ for } n\geq k\geq 0,\\[15pt] 0 & \text{ otherwise, }\end{cases}$$ are the $q$-binomial coefficients.
\end{proposition}

\begin{proof}[Proof of Theorem~\ref{main-URG}]
The plan is to show that the right-hand side of \eqref{id:Zeng-Zhang} is exactly the generating function of $\bmajBAST$ over $\URG(n,k)$, then Theorem~\ref{main-URG} follows immediately. More precisely, for a fixed $i,~1\le i\le k$, we aim to show that
$$\sum_{\widebar{p}\in\widebar{\SS}_n^{k-i,i-1}}q^{\bmajBAST(\theta(\widebar{p}))}=q^{k(k-i)}\cdot \qbin{n-i}{k-i}{q}\cdot A_q(n,i-1).$$
Then we are done by noting that $\bigsqcup_{i=1}^k\widebar{\SS}_n^{k-i,i-1}=\theta^{-1}(\URG(n,k))$.

For a given $\widebar{p}\in\widebar{\SS}_n^{k-i,i-1}$, let $\theta(\widebar{p})=w$, and suppose the active bars of $\widebar{p}$ are respectively the $a_1$-th, $a_2$-th, $\ldots~$, $a_{k-i}$-th bar among all bars, counting from left to right. We describe below the relations between various statistics, with respect to either $\widebar{p}$ or $w$. All of them should be readily checked.
\begin{itemize}
	\item The $j$-th active bar contributes $k-a_j$ to $\bMAJ(w)$.
	\item If we denote the number of empty slots to the left of the $j$-th active bar by $e_j$, then the number of letters to the left of the $j$-th active bar is $a_j+e_j$.
	\item Each fixed bar contritutes $1$ to $\des(\widebar{p})=\bndes(w)$, hence $\des(p)=\bndes(w)=i-1$.
	\item The number of letters to the left of each fixed bar adds up to $\MAJ(p)$.
	\item The number of letters to the left of each bar (either fixed or active), adds up to $\MIL(w)$.
	\item $p_n=\pl(w)$.
\end{itemize}

Now to finish the proof, it suffices to interprete each factor appropriately. We see that
\begin{align*}
q^{k(k-i)} &=q^{\sum_{j=1}^{k-i}(a_j+k-a_j)}=q^{\bMAJ(w)}\cdot q^{\sum_{j=1}^{k-i}a_j},\\
\qbin{n-i}{k-i}{q} &= \sum q^{\sum_{j=1}^{k-i}e_j},\\
A_q(n,i-1) &=\sum_{p\in\SS_n^{i-1}}q^{\BAST(p)}=\sum_{p\in\SS_n^{i-1}}q^{\MAJ(p)+p_n+\des(p)-n},
\end{align*}
where the sum in the second identity runs over all barred permutations with the same base permutation $p\in\SS_n^{i-1}$, and $k-i$ active bars inserted into the $n-i$ available empty slots, and we have applied \eqref{des-maj=des-bast} and \eqref{231-213} to derive the third one. We multiply them together to see that
$$q^{k(k-i)}\cdot \qbin{n-i}{k-i}{q}\cdot A_q(n,i-1)=\sum_{\widebar{p}\in\widebar{\SS}_n^{k-i,i-1}}q^{(\bMAJ+\MIL+\pl+\bndes)\:w-n}=\sum_{\widebar{p}\in\widebar{\SS}_n^{k-i,i-1}}q^{\bmajBAST(\theta(\widebar{p}))}.$$
The proof is now completed.
\end{proof}

Thanks to the involution $\psi$ constructed in the last section, we actually have the following stronger version.
\begin{theorem}\label{main-URG-st}
Statistics $(\Asc, \bmajMIL)$ and
$(\Asc, \bmajBAST)$ have the same joint distribution on $\URG(n,k)$ for all $n \ge k\ge 1$.
\end{theorem}
\begin{proof}
It suffices to give a bijection $\eta$ on $\URG(n,k)$ which preserves $\Asc$ and sends $\bmajMIL$ to $\bmajBAST$.
Given $w \in \URG(n,k)$, we construct $u=\eta(w)$ as follows.
Let $\widebar{p}=\theta^{-1}(w)$ with
$\widebar{p}\in\widebar{\SS}_n^{k-i,i-1}$ and $e_j$ being the  number of empty slots to the left of the $j$-th active bar of $p$.
Assume that $q=\psi\circ \phi^{-1}(p)$ and let $\widebar{q}$ be the barred permutation obtained
from $q$ by the following two steps.
\begin{itemize}
  \item Step $1$. Insert $i-1$ fixed bars at the $i-1$ descents of $q$.
  \item Step $2$. Insert $k-i$ active bars at the accents of $q$,
                  with $e_j$ empty slots to the left of the $j$-th active bar of $q$ for $1 \leq j \leq k-i$. 
\end{itemize}
Finally, let $u=\eta(w)=\theta(\widebar{q})$.

By the properties of the involutions $\phi$ and $\psi$, we have  $\widebar{q}\in\widebar{\SS}_n^{k-i,i-1}$ and  $\Id(p)=\Id(q)$.
Then,  it is not hard to check that $\eta$ is a bijection on $\URG(n,k)$ and $\Asc(w)=\Asc(u)$.

From the proof of Theorem \ref{main-URG}, we may see that
\begin{align*}
  \bmajMIL(w) & =k(k-i)+ \sum_{j=1}^{k-i}e_j+\MAJ(p),\\
  \bmajBAST(u)& =k(k-i)+ \sum_{j=1}^{k-i}e_j+\BAST(q).
\end{align*}
Notice that $\MAJ(p)=\BAST(q)$, which leads to $\bmajMIL(w)=\bmajBAST(u)$, as desired.
\end{proof}
\subsection{Words and ordered multiset partitions}\label{subsec:multiset}
When we view permutations as the rearrangement class of $[n]$, then naturally the next thing to consider is to allow repetition of letters. This leads us to the rearrangement class $R(w)$ of a given word $w$, and the notion of ordered multiset partitions (see \cite{Rh}). 

A sequence of nonnegative integers, say $\beta=(\beta_1,\beta_2,\ldots)$, which satisfies $\beta_1+\beta_2+\cdots=n$, is called \emph{a weak composition of $n$}, and we denote this as $\beta\models_0 n$. For any weak composition $\beta=(\beta_1,\beta_2,\ldots)$, an \emph{ordered multiset partition of weight $\beta$ with $k$ blocks} is a sequence $\mu=B_1|\cdots|B_k$ of nonempty {\bf sets} such that $\{1^{\beta_1},2^{\beta_2},\cdots\}$ is the multiset union $B_1\cup \cdots \cup B_k$. The total number of letters $\sum_{i=1}^k|B_i|$ is called the \emph{size} of $\mu$, the sets $B_i$ are called the \emph{blocks} of $\mu$. Note that we do not allow repetition of letters within blocks in an ordered multiset partition, see \cite{HRW} and \cite{Rh} for the motivations behind this definition. Given integers $n\ge k$, a weak composition $\beta\models_0 n$, we use the following notation (see \cite{Rh}):
\begin{align*}
\OP_{\beta,k} :=\{\text{all ordered multiset partitions of weight $\beta$ with $k$ blocks}\}.
\end{align*}

Since all the statistics we consider here are concerned with the relative greatness in numerical values between letters, without loss of generality, we can assume $\beta$ is a composition, i.e., all components of $\beta$ are positive integers. For a given ordered multiset partition of weight $\beta=(\beta_1,\ldots,\beta_s)$, say $\mu=B_1|B_2|\cdots|B_k$, we form a word $\iota(\mu)=w:=u_1\cdots u_s$ by concatenating the longest subwords $u_i=j_1\cdots j_{\beta_i}$ with $j_1<\cdots<j_{\beta_i}$, such that $i\in B_{j_1},\ldots,B_{j_{\beta_i}}$. For example, let $\mu=23|12|1\in\OP_{(2,2,1),3}$, then $\iota(\mu)=23121\in\URG(5,3)$. This map $\iota$ is clearly seen to be a bijection between $\OP_{\beta,k}$, and the following subset of $\URG(n,k)$:
\begin{align*}
\URG(\beta,k):=\{w\in\URG(n,k) : [n]\subset \Asc(w)\cup\{\beta_1,\beta_1+\beta_2,\ldots,\beta_1+\cdots+\beta_{s-1}\}\}.
\end{align*}
Now we can extend all the statistics defined on $\URG(n,k)$ to the ordered multiset partitions $\OP_{\beta,k}$. In particular, we let
\begin{align*}
\Asc(\mu):=\Asc(\iota(\mu)),\; \bmajMIL(\mu):=\bmajMIL(\iota(\mu)),\; \bmajBAST(\mu):=\bmajBAST(\iota(\mu)),
\end{align*}
and the following result follows immediately.
\begin{theorem}\label{main-betaURG}
Statistics $(\Asc,\bmajMIL)$ and $(\Asc,\bmajBAST)$ have the same joint distribution on $\OP_{\beta,k}$ for any $\beta\models_0 n$ and $n\ge k\ge 1$.
\end{theorem}
\begin{proof}
From the proof of Theorem~\ref{main-URG-st}, the bijection $\eta: \URG(n,k)\rightarrow\URG(n,k)$ preserves the set-valued statistic $\Asc$. Consequently, for any $w\in\URG(\beta,k)$, $\eta(w)$ remains in $\URG(\beta,k)$. Now we see $\hat{\eta}:=\iota^{-1}\circ\eta\circ\iota$ is a well-defined bijection from $\OP_{\beta,k}$ to itself, and for any $\mu\in\OP_{\beta,k}$, we have
\begin{align*}
\Asc(\hat{\eta}(\mu))&=\Asc(\eta\circ\iota(\mu))=\Asc(\iota(\mu))=\Asc(\mu),\\
\bmajBAST(\hat{\eta}(\mu))&=\bmajBAST(\eta\circ\iota(\mu))=\bmajMIL(\iota(\mu))=\bmajMIL(\mu).
\end{align*}
\end{proof}

In order to make connection with $\Val_{n,k}(x;q,t)$ given in the Introduction, we actually need the ``smaller'' statistics, i.e., those whose distributions on $\URG(n,k)$ generate $[k]_q!\cdot\widetilde{S}_q(n,k)$ instead of $[k]_q!\cdot S_q(n,k)$. Let us first recall that $\INV$ and $\MAJ$ are equidistributed over $\SS_n$, a classical result due to MacMahon \cite{Ma}. Wilson \cite{Wils} extended both $\INV$ and $\MAJ$ to statistics $\inv$ and $\maj$ on ordered multiset partitions and obtained the following analogous equidistribution result.
\begin{theorem}[Theorem~4.0.1 in \cite{Wils}]\label{inv-maj on OMP}
For all $n\ge k\ge 1$ and any composition $\beta\models n$,
\begin{align*}
\sum_{\mu\in\OP_{\beta,k}}q^{\inv(\mu)}=\sum_{\mu\in\OP_{\beta,k}}q^{\maj(\mu)}.
\end{align*}
\end{theorem}

We recall the definition of $\maj$ from \cite{Wils}; for $\inv$ see Remark~\ref{rmk:inv}. Suppose $\mu=B_1|B_2|\cdots$ is an ordered multiset partition, let $\sigma=\sigma(\mu)$ be the word obtained by writing letters in each block in decreasing order and deleting the bars. Next we recursively form another word $w$ by setting $w_0=0$ and $w_i=w_{i-1}+\chi(\text{$\sigma_i$ is minimal in its block in $\mu$})$ for each $i>0$. Then we set
\begin{align*}
\maj(\mu)=\sum_{i: \sigma_i>\sigma_{i+1}}w_i.
\end{align*}
We take $\mu=124|35|12$ for example. First we get $\sigma(\mu)=4215321$ and $w=00011223$, so $\maj(\mu)=w_1+w_2+w_4+w_5+w_6=0+0+1+2+2=5$.

We need to make one more observation that links $\maj$ with $\bmajMIL$, before we can prove Theorem~\ref{main-OMP}.
\begin{proposition}\label{maj-bmajMIL}
For all $n\ge k\ge 1$, any composition $\beta\models n$ and any $\mu\in\OP_{\beta,k}$, we have $$\maj(\mu)+\binom{k}{2}=\bmajMIL(\mu).$$
\end{proposition}
\begin{proof}
Suppose $\sigma(\mu)=\sigma_1\cdots\sigma_n$, then for $i=1,\ldots,n$, we examine the contributions from the $i$-th position to both sides of the equation. We discuss by three cases.
\begin{itemize}
	\item $\sigma_i>\sigma_{i+1}$ and they both belong to the $j$-th block of $\mu$. Then the contribution to $\maj(\mu)$ is $w_i=j-1$ and the contribution to $\MIL$ is also $j-1$.
	\item $\sigma_i>\sigma_{i+1}$ and they belong to the $j$-th and the $(j+1)$-th blocks of $\mu$, respectively. Then we see all letters in the $j$-th block are larger than all letters in the $(j+1)$-th block, since in $\sigma$, letters in the same block are ordered decreasingly. Therefore the contribution to $\maj(\mu)$ is $w_i=j$ and to $\binom{k}{2}$ is $j-1$. For the right hand side, the contribution to $\bMAJ(\mu)$ is $j$ and to $\MIL$ is $j-1$.
	\item $\sigma_i\le\sigma_{i+1}$. Then they must belong to different blocks, say $\sigma_i$ is in the $j$-th block. This contributes $0$ to $\maj(\mu)$ and $j-1$ to $\binom{k}{2}$. For the right hand side, it contributes $0$ to $\bMAJ$ and $j-1$ to $\MIL$.
\end{itemize}
In all three cases above, we see the contributions to both sides of the equation are the same. The proof is completed by summing up all the contributions.
\end{proof}

\begin{proof}[Proof of Theorem~\ref{main-OMP}]
In view of Remark~\ref{rmk:inv}, we only need to prove \eqref{id2:big Val-q-0}. The second equality in \eqref{id2:big Val-q-0} follows directly from applying Theorem~\ref{main-betaURG}. For the first one, it suffices to combine Proposition~\ref{maj-bmajMIL} and Theorem~\ref{inv-maj on OMP}.
\end{proof}

\section{Final remarks}\label{sec:conclusion}
Note that $\STAT(\pi)=\BAST(\pi^{\rc})$, and $\des(\pi)=\des(\pi^{\rc})$ for any permutation $\pi$. Applying this to the first identity in Corollary~\ref{main-Sn}, we obtain the equidistribution directly between $\MAJ$ and $\STAT$, albeit on different subsets of $\SS_n$:
\begin{align}\label{id:maj-stat}
\sum\limits_{\sigma \in \SS_n^{k-1}(1\underline{32})}
    q^{\MAJ(\sigma)}
=
\sum\limits_{\sigma \in \SS_n^{k-1}(\underline{21}3)}
    q^{\STAT(\sigma)}.
\end{align}
The natural desire to refine \eqref{id:maj-stat} with further statistics, in a sense to match the result of, say Corollary~\ref{main-bijectionper}, propelled us to eventually discover Theorem~\ref{thm:stat-bast}. Although the involution $\psi$ does preserve more statistics, it unfortunately is not closed on the set $\SS_n(1\underline{32})$. For instance, $4312\in\SS_4(1\underline{32})$ but $\psi(4312)=4132\not\in\SS_4(1\underline{32})$. As a matter of fact, we do not have a simple way to characterize the image set $\psi(\SS_n(1\underline{32}))$, due to the somewhat involved construction of $\psi$.

Another thing one might be wondering, after comparing Corollary~\ref{main-bijectionper} with Theorem~\ref{thm:FHV}, is that if the stronger property, namely the pairs $(\MAJ,\BAST)$ and $(\BAST,\MAJ)$ are equidistributed on $\SS_n(1\underline{32})$ would hold. This is not ture, as witnessed by Table~\ref{MAJ-BAST} below.

To continue the theme of this paper, one might want to explore equidistributions of similar nature between $\MAJ$, $\BAST$ and $\STAT$ over other pattern-avoiding subsets of $\SS_n$. An ideal candidate would be $\SS_n(\tau)$, $\tau\in\SS_3$, which is well-known to be enumerated by the Catalan numbers. Amini \cite{Am} launched a systematic investigation in this direction. We have also accumulated some initial results to this end and they will be discussed in our future work \cite{CF}.

\begin{table}[htbp]
\centering \caption{The joint distribution of four statistics on $\SS_4(1\underline{32})$}
\begin{tabular}{|c||c|c|c|c|c|c|c|}
\hline
$\SS_4(1\underline{32})$ & $\Db$ & $\Id$ & $\MAJ$ & $\BAST$\\
\hline
1234 & \{1\} & $\emptyset$ & 0 & 0\\
\hline
2134 & \{1,2\} & \{1\} & 1 & 2\\
\hline
2314 & \{1,2\} & \{1\} & 2 & 3\\
\hline
2341 & \{1,2\} & \{1\} & 3 & 1\\
\hline
2413 & \{1,2\} & \{1,3\} & 2 & 2\\
\hline
3124 & \{1,3\} & \{2\} & 1 & 2\\
\hline
3412 & \{1,3\} & \{2\} & 2 & 1\\
\hline
3214 & \{1,2,3\} & \{1,2\} & 3 & 5\\
\hline
3241 & \{1,2,3\} & \{1,2\} & 4 & 3\\
\hline
3421 & \{1,2,3\} & \{1,2\} & 5 & 4\\
\hline
4123 & \{1,4\} & \{3\} & 1 & 1\\
\hline
4213 & \{1,2,4\} & \{1,3\} & 3 & 4\\
\hline
4231 & \{1,2,4\} & \{1,3\} & 4 & 3\\
\hline
4312 & \{1,3,4\} & \{2,3\} & 3 & 3\\
\hline
4321 & \{1,2,3,4\} & \{1,2,3\} & 6 & 6\\
\hline
\end{tabular}
\label{MAJ-BAST}
\end{table}

\section*{Acknowledgement}
The first author was supported by Natural Science Foundation Project
of Tianjin Municipal Education Committee (No.~2017KJ243,~No.~2018KJ193)
and the National Natural Science Foundation of China (No.~11701420).
The second author was supported by the Fundamental Research Funds for the Central Universities (No.~2018CDXYST0024) and the National Natural Science Foundation of China (No.~11501061).

\end{document}